\documentclass{article}
\usepackage[utf8]{inputenc}

\usepackage{appendix}

 \usepackage[utf8]{inputenc}
\usepackage{amsmath,amsfonts,amssymb,amsthm}
\usepackage{bbm}
\usepackage[top=2cm, bottom=2cm,left=2cm,right=2cm]{geometry}

\usepackage[export]{adjustbox}%to center the diagram
\usepackage{bm}
\usepackage{graphicx}
\usepackage{svg}
\usepackage{pgf,tikz,pgfplots}
\usepackage{subfiles}

\usepackage{appendix}

\newtheorem{definition}{Definition}
\newtheorem{lemma}{Lemma}
\newtheorem{theorem}{Theorem}
\newtheorem{corollary}{Corollary}

\usepackage{hyperref}
\usepackage{aliascnt}

\hypersetup{%
  colorlinks=true,% hyperlinks will be black
  linkcolor=blue,
  citecolor=blue,
  linkbordercolor=red,% hyperlink borders will be red
  pdfborderstyle={/S/U/W 1}% border style will be underline of width 1pt
}

\usepackage{enumitem}

\usepackage{cleveref}
\usepackage{caption}
\usepackage{subcaption}
\usepackage[ruled,vlined]{algorithm2e}
\usepackage{epstopdf}

\tikzset{
point/.style={circle,draw=black,inner sep=0pt,minimum size=3pt}
}
\pgfplotsset{
    soldot/.style={color=blue,only marks,mark=*}
    }

\author{}
\pgfplotsset{compat=1.14}
\usepackage{mathrsfs}
\usetikzlibrary{arrows}

\renewcommand{\[}{\begin{equation}}
\renewcommand{\]}{\end{equation}}

\newcommand{\T}{\mathcal{T}}

\renewcommand{\r}{\mathbf{r}}

\newcommand{\cC}{\mathcal{C}}
\renewcommand{\H}{\mathcal{H}}
\newcommand{\X}{\mathcal{X}}
\renewcommand{\S}{\mathcal{S}}
\newcommand{\HH}{\mathbb{H}}
\newcommand{\cu}{\check{u}}

\newcommand{\VV}{\mathbb{V}}
\newcommand{\WW}{\mathbb{W}}

\newcommand{\diag}{\mathrm{diag}}
\renewcommand{\v}{\mathbf{v}}
\renewcommand{\u}{\mathbf{u}}
\newcommand{\cuu}{\check{\mathbf{u}}}

\newcommand{\pp}{\mathbf{p}}

\newcommand{\Id}{\mathrm{Id}}

\newcommand{\z}{\mathbf{z}}

\newcommand{\cI}{\check{I}}
\newcommand{\calpha}{\check{\alpha}}

\newcommand{\cV}{\V^m_{\calpha,r}(\cI)}
\newcommand{\cP}{\check{P}^m_{\calpha,r}}
\newcommand{\cpi}{\check{\pi}^m_{\calpha,r}}
\newcommand{\cp}{\check{\varphi}}
\newcommand{\cpsi}{\check{\psi}}
\newcommand{\ch}{\check{h}}
\newcommand{\cO}{\mathcal{O}^{m}_{\calpha,w,r}}

\newcommand{\cv}{\check{v}}
\newcommand{\cPi}{\check{\Pi}^m_{\calpha,\r}}
\newcommand{\cVV}{\VV^{m}_{\calpha,\r}(\cI)}
\newcommand{\Ik}{I_{k_0}}
\newcommand{\Vk}{\V^m_{\alpha,r}(\Ik)}

\renewcommand{\Pr}{P^m_{\alpha,r}}
\newcommand{\cQ}{\mathcal{Q}^{m}_{\calpha,w,r}(\cI)}
\newcommand{\cQQ}{\mathbb{Q}^{m}_{\calpha,S,\r}(\cI)}

\newcommand{\cL}{\check{\mathscr{L}}^{m}_{\calpha,r}}

\newcommand{\ucl}{\tilde{u}}
\newcommand{\cw}{\check{w}}

\newcommand{\w}{\mathbf{w}}

\usepackage{stmaryrd}

 \newcommand{\bb}[1]{\left\llbracket#1\right\rrbracket}

\newcommand{\tA}{\tilde{A}}

\newcommand{\C}{\tilde{C}}

\renewcommand{\c}{\mathbf{c}}

\newcommand{\g}{\mathbf{g}}

\newcommand{\Ru}{\mathcal{R}\u}

\renewcommand{\L}{\mathscr{L}}

\newcommand{\norm}[1]{\left\lVert#1\right\rVert}

\newcommand{\abs}[1]{\left|#1\right|}

\renewcommand{\C}{\mathcal{C}}
\newcommand{\N}{\mathcal{N}}
\newcommand{\M}{\mathcal{M}}
\renewcommand{\P}{\mathcal{P}}
\newcommand{\V}{\mathcal{V}}

\newcommand{\p}{\varphi}

\usepackage[utf8]{inputenc}
\newcommand{\e}{\mathbf{e}}
\newcommand{\q}{\mathbf{q}}

\newcommand{\Ep}{\mathcal{E}^{m,+}_{\calpha,r}}

\newcommand{\Es}{\mathcal{E}^{m,s}_{\calpha,r}}
\newcommand{\Esp}{\mathcal{E}^{m,s'}_{\calpha,r}}

\newcommand{\ddx}{\frac{\mathrm{d}}{\mathrm{d}x}}

\newcommand{\hpm}{\hat{\p}_-}

\usepackage{cases}
\usepackage[overload]{empheq}

\title{A Unified Immersed Finite Element Error Analysis \\ for One-Dimensional Interface Problems}
\newcommand{\comment}[1]{}

\newcommand{\commentout}[1]{{}} % for large block comments

\author{Slimane Adjerid\thanks{adjerids@math.vt.edu} \and Tao Lin\thanks{tlin@vt.edu} 
\and Haroun Meghaichi\thanks{haroun@vt.edu}}

\date{Department of Mathematics, Virginia Tech
}
\begin{document}

\maketitle

\begin{abstract}

    It has been noted that the traditional scaling argument cannot be directly applied to the error analysis of immersed finite elements (IFE) because, in general, the spaces on the reference element associated with the IFE spaces on different interface elements via the standard affine mapping are not the same. By analyzing a mapping from the involved Sobolev space to the IFE space, this article is able to extend the scaling argument framework to the error estimation for the approximation capability of a class of IFE spaces in one spatial dimension. As demonstrations of the versatility of this unified error analysis framework, the manuscript applies the proposed scaling argument to obtain optimal IFE error estimates for a typical first-order linear hyperbolic interface problem, a second-order elliptic interface problem, and the fourth-order Euler-Bernoulli beam interface problem, respectively.

    % Numerical examples are provided to illustrate the features of involved IFE spaces.
\end{abstract}

{\bf keywords:}
Immersed finite elements, error analysis, Bramble Hilbert lemma, scaling argument.

\section{Introduction}

Partial differential equations with discontinuous coefficients arise in many areas of sciences and engineering such as heat transfer, acoustics, structural mechanics, and electromagnetism. The discontinuity of the coefficients results in multiple challenges in the design and the analysis of numerical methods and it is an active area of research in the communities of finite element, finite volume, as well as finite difference method.

The immersed finite element (IFE) methods can use an interface independent mesh to solve an interface problem. Many publications were about IFE methods using either linear \cite{liImmersedInterfaceMethod1998,liImmersedFiniteElement2004,liNewCartesianGrid2003b}, bilinear \cite{X.He_T.Lin_Y.Lin,linRectangularImmersedFinite2001}, or trilinear polynomials \cite{guoImmersedFiniteElement2020, SVallaghe_TPapadopoulo_TriLinear_IFE}. IFE methods have been applied to a variety of problems, such as parabolic interface problems \cite{heImmersedFiniteElement2013,linImmersedFiniteElement2013,linOptimalErrorBounds2020}, hyperbolic interface problems \cite{adjeridErrorEstimatesImmersed2020}, the acoustic interface problem \cite{adjeridHigherOrderImmersed2014,moonImmersedDiscontinuousGalerkin2016} and Stokes and Navier-Stokes interface problems \cite{adjeridImmersedDiscontinuousFinite2019,2021ChengZhang,2021JonesZhang,2022WangZhangZhuang}. IFE methods with higher degree polynomials have also been explored \cite{adjeridHigherDegreeImmersed2018,
adjeridHIGHDEGREEIMMERSED2017,
adjeridPthDegreeImmersed2009,
2021ChengZhang,
guoHigherDegreeImmersed2019,
2019LinLinZhuan_Helmholtz1}. In particular, Adjerid and Lin \cite{adjeridPthDegreeImmersed2009} constructed IFE spaces of arbitrary degree and analyzed their approximation capabilities.

In \cite{adjeridHigherOrderImmersed2014,moonImmersedDiscontinuousGalerkin2016}, Adjerid and Moon discussed IFE methods for the following acoustic interface problem
\begin{equation}\begin{cases}
p_t(x,t)=\rho(x) c(x)^2 v_x(x,t),& x\in(a,\alpha)\cup (\alpha,b),\\
\rho(x) v_t(x,t)=p_x(x,t),&x\in(a,\alpha)\cup (\alpha,b),\\
[v]_{x=\alpha}=[p]_{x=\alpha}=0.
\end{cases}\label{eqn:acoustic_acoustic_problem} \end{equation}
where $\rho,c$ are equal to $\rho_+,c_+$ on interval $(\alpha,b)$ and to $\rho_-,c_-$ on $(a,\alpha)$. Assuming that the exact solution $(u, p)$ has sufficient regularity in $(a, \alpha)$ and $(\alpha, b)$, respectively, we can follow the idea in
\cite{Lombard_Piraux_2001_1D} to show that the exact solution satisfies the following so-called extended jump conditions:
\begin{equation}\frac{\partial^k}{\partial x^k} p(\alpha^+,t)=r^p_k
\frac{\partial^k}{\partial x^k} p(\alpha^-,t),\qquad \frac{\partial^k}{\partial x^k} v(\alpha^+,t)=r^v_k
\frac{\partial^k}{\partial x^k} v(\alpha^-,t),\qquad k=0,1,\dots,\label{eqn:extended_conditions_acoustic}
\end{equation}
for certain positive constants $r^p_k$ and $r^v_k$. In \cite{adjeridHigherOrderImmersed2014,moonImmersedDiscontinuousGalerkin2016}, IFE spaces based on polynomials of degree up to $4$ were developed with these extended jump conditions, and these IFE spaces were used with a discontinuous Galerkin (DG) method to solve the above acoustic interface problem with pertinent initial and boundary conditions. Numerical examples presented in \cite{adjeridHigherOrderImmersed2014,moonImmersedDiscontinuousGalerkin2016} demonstrated the optimal convergence of this DG IFE method, but we have not seen any error analysis about it in the related literature.

Extended jump conditions have also been used in the development of higher degree IFE spaces for solving other interface problems
\cite{adjeridHigherDegreeImmersed2018,
adjeridHIGHDEGREEIMMERSED2017,
adjeridPthDegreeImmersed2009,
2021ChengZhang,
guoHigherDegreeImmersed2019,
2019LinLinZhuan_Helmholtz1}. This motivates us to look for a unified framework for the error analysis for methods based on IFE spaces constructed with the extended jump conditions such as those in \eqref{eqn:extended_conditions_acoustic}. As an initial effort, our focus here is on one-dimensional interface problems.

One challenge in error estimation for IFE methods is that the scaling argument commonly used in error estimation for traditional finite element methods cannot be directly applied. In the standard scaling argument, local finite element spaces on elements in a sequence of meshes are mapped to the same finite element space on the reference element via an affine transformation. However, the same affine transformation will map the local IFE spaces on interface elements in a sequence of meshes to different IFE spaces on the reference element because of the variation of interface location in the reference element, see the illustration in \autoref{fig:phys_ref}. A straightforward application of the scaling argument to the analysis of the approximation capability of an IFE space will result in error bounds of a form $C(\calpha)h^r$, i.e., the constant factor $C(\calpha)$ in the derived
error bounds depend on the location of the interface in the reference element, and this kind of error bounds cannot be used to show the convergence of the related IFE method unless one can show that the constant factor $C(\calpha)$ is bounded for all $\calpha$ in the reference element, which, to our best knowledge, is difficult to establish. Alternative analysis techniques such as multi-point Taylor expansions are used \cite{adjeridPthDegreeImmersed2009} which becomes awkward for higher degree IFE spaces, particularly so for higher degree IFE spaces in higher dimension. To circumvent this predicament of the classical scaling argument, we introduce a mapping between the related Sobolev space and the IFE space by using weighted averages
of the derivatives in terms of the coefficients in the jump conditions. We show that the Sobolev norm of the error of this mapping can be bounded by the related Sobolev semi-norm. This essential property enables us to establish a Bramble-Hilbert type lemma for the IFE spaces, and, to our best knowledge, this is the first result that makes the scaling argument applicable in the error analysis of a class of IFE methods. For demonstrating the versatility of this unified error analysis framework, we apply it to establish, for the first time,  the optimal approximation capability of the IFE space designed for the acoustic interface problem \eqref{eqn:acoustic_acoustic_problem}. Similarly, we apply this immersed scaling argument to the IFE space designed for an elliptic interface problem considered in \cite{adjeridPthDegreeImmersed2009} as well as the IFE space for the Euler-Bernoulli Beam interface problem considered in  \cite{linErrorAnalysisImmersed2017,linImmersedFiniteElement2011,wangHermiteCubicImmersed2005} leading to much simpler and elegant proofs.

\commentout{
        In this paper, we provide a detailed proof for the convergence rate of the IFE-DG scheme for a general class of problems including the acoustic interface problem \eqref{eqn:acoustic_acoustic_problem}.  The novelty here is the use of the scaling argument, that is, we construct our IFE spaces and local projection operators on the reference element $[0,1]$ and map them to the physical element that contains the interface point. To our knowledge, the scaling argument, although being a classical tool in FEM, has not been used before to analyze the properties of immersed finite element method. The main concern is that $\calpha$, the  relative position of the interface point in the interface element changes when the mesh is refined as \autoref{fig:phys_ref} shows because of different relative positions of the interface point $\alpha$ in interface elements.
}

\begin{figure}
    \centering
    \includegraphics{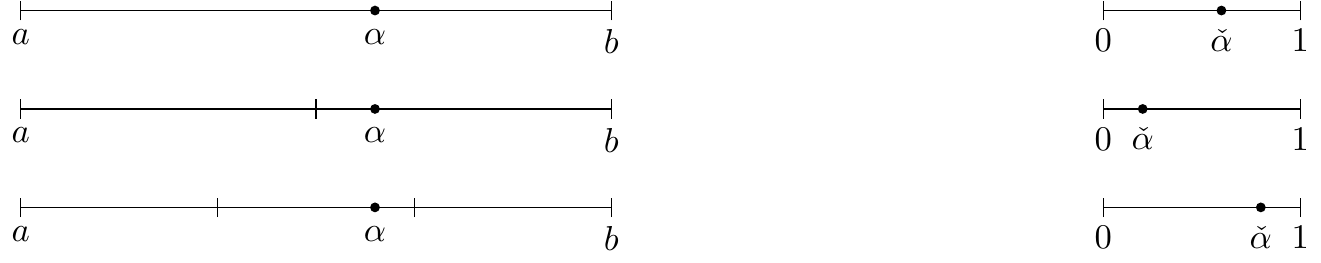}
    \caption{The relative position of the interface (on the right) changes as the we refine the mesh (on the left).}
    \label{fig:phys_ref}
\end{figure}

\commentout{
      In order for the scaling argument to yield an optimal convergence rate, we need to show that our estimates are independent of $\calpha$, namely, we have to prove that the necessary projection operators norms are bounded uniformly for all $\calpha\in(0,1)$. In addition to that, we need an immersed version of the Bramble-Hilbert lemma for functions that have jump discontinuities of the form \eqref{eqn:extended_conditions_acoustic}. This constitutes a sizeable overhead in the analysis. However, it makes the use of the classical techniques, such as Radau projection, Lobatto projection and Céa's lemma applicable to interface problems with some minor changes. Furthermore, our general approach applies to the elliptic interface problem $-(\beta u')'=f$, where $\beta$ is piecewise constant, and to the parabolic interface problem $u_t=(\beta u')'$ fairly easily.
}
  The paper is organized as follows. In \autoref{sec:notation}, we introduce the notation and spaces used in the rest of the paper. In \autoref{sec:properties_of_v}, we restrict ourselves to study of the IFE  functions on the interval $[0,1]$, we show that they have similar properties to polynomials, for example, they both have the same maximum number of roots, they both admit a Lagrange basis and they both satisfy an inverse and a trace inequality. In \autoref{sec:Bramble_Hilbert}, we define the notion of uniformly bounded RIFE operators and how the scaling argument is applicable using an immersed Bramble-Hilbert lemma. In \autoref{sec:one_D_class}, we study the convergence of the DG-IFE method for the acoustic interface problem \eqref{eqn:acoustic_acoustic_problem}.
  In \autoref{sec:Already_established_results}, we give shorter and simpler
  proofs for the optimal convergence of IFE methods for the second-order elliptic interface problem as well as the
  fourth-order Euler-Bernoulli beam interface problem.

\section{Preliminaries}
\label{sec:notation}

Throughout the article, we will consider a bounded open interval $I=(a,b)$ with $\abs{a}, \abs{b} < \infty$, and let $\alpha\in I$ be  the interface point dividing $I$ into two open intervals $I^{-}=(a,\alpha), I^+=(\alpha,b)$. This convention extends to any other open interval $B\subseteq \mathbb{R}$ with $B^-=B\cap(-\infty,\alpha)$ and $B^+=B\cap(\alpha,\infty)$. For every bounded open interval $B$ not containing $\alpha$, let $W^{m,p}(B)$ be the usual Sobolev space on $B$ equipped with the norm $\norm{\cdot}_{m,p,B}$ and the seminorm $|\cdot|_{m,p,B}$. We are particularly interested in the case of $p=2$ corresponding to the Hilbert space $H^m(B)=W^{m,2}(B)$, and we will use $\norm{\cdot}_{m,B}$ and $|\cdot|_{m,B}$ to denote $\norm{\cdot}_{m,2,B}, |\cdot|_{m,2,B}$ respectively for convenience. We will use $(\cdot,\cdot)_{B}$ and $(\cdot,\cdot)_{w,B}$ to denote the classical and the weighted $L^2$ inner product defined as
$$(f,g)_{B}= \int_{B}f(x)g(x)\ dx,\qquad  (f,g)_{w,B}= \int_{B}w(x)f(x)g(x)\ dx,\qquad w(x)>0,\ \forall\ x\in B. $$
Given a positive finite sequence $r=(r_i)_{i=0}^{\tilde m}, ~{\tilde m} \geq 0$
and an open interval $B$ containing $\alpha$, we introduce the following piecewise Sobolev space:
\begin{align}
\H^{m+1}_{\alpha,r}(B)&=\left\{u\mid u_{\mid B^\pm}\in H^{m{+1}}(B^{\pm}), u^{(k)}(\alpha^+)=r_k u^{(k)}(\alpha^-),\ \forall k=0,1,\dots,m\right\}, ~0 \leq m \leq {\tilde m}. \label{eqn:defintion_of_H}
%\\
%\C^{m}_{\alpha,r}(B)&=\left\{u\mid u_{\mid B^\pm}\in C^{m}(B^{\pm}), u^{(k)}(\alpha^+)=r_k u^{(k)}(\alpha),\ \forall k=0,1,\dots,m\right\},\label{eqn:defintion_of_C}\\
%\X_{\alpha,r}^{m}(B)&=\H_{\alpha,r}^{m}(B)\cap \C^m_{\alpha,r}(B). \label{eqn:defintion_of_X}
\end{align}
\comment{ We need $u\in H^{k+1}(B^{\pm})$ in order for the trace $u^{(k)}(\alpha^{\pm})$ to exist according to the trace theorem. For example, take $u(x)=x^{2/3}$ on $[0,1]$, we have $u\in H^1([0,1])$ and $u'(0^{+})=\infty$.
}
\comment{
The two spaces are the same in the following sense: $H^{m+1}(B^\pm)\hookrightarrow C^m(B^\pm)$. Therefore, given $(u,u^+)\in H^{m+1}(B^-)\times H^{m+1}(B^+)$, there is $(v^-,v^+)\in C^{m}(B^-)\times C^{m}(B^+)$ such that $u^\pm=v^\pm$ a.e. Furthermore, the trace theorem guarantees that $v^{\pm}(\alpha^\mp)=\text{Tr}[u^{\pm}](\alpha^\pm)$.
}
The norms, semi-norms and the inner product that we will use on $\H^{m+1}_{\alpha,r}(B)$ are %$\C^{m}_{\alpha,r}(B)$ and $\X_{\alpha,r}^{m}(B)$ are
$$
 \norm{\cdot}_{m+1,B}=\sqrt{\norm{\cdot}_{m+1,B^-}^2+\norm{\cdot}_{m+1,B^+}^2},
~~\abs{\cdot}_{m,B}=\sqrt{\abs{\cdot}_{m,B^-}^2+\abs{\cdot}_{m,B^+}^2}, ~~\left(f,g\right)_{w,B}=\left(f,g\right)_{w,B^-}+\left(f,g\right)_{w,B^+}.
$$
We note, by the Sobolev embedding theory, that $\H^{m+1}_{\alpha,r}(B)$ is a subspace of
\begin{align}
\C^{m}_{\alpha,r}(\overline{B})&=\left\{u\mid u_{\mid B^\pm}\in C^{m}(\overline{B^{\pm}}), u^{(k)}(\alpha^+)=r_k u^{(k)}(\alpha^-),\ \forall k=0,1,\dots,m\right\}, ~0 \leq m \leq {\tilde m}. \label{eqn:defintion_of_C}\
\end{align}

\commentout{
 Throughout the paper, we will use $u'$ to denote the piecewise derivative of $u$, i.e., given $u\in \X^m_{\alpha,r}(B)$, $u_{\mid B^\pm}'$ is the classical derivative of $u_{\mid B^\pm}$ and it should not be confused with the distributional derivative of $u$:
\begin{equation}
    u'(x)=\begin{cases}\frac{du}{dx}(x),& x\in B\backslash\{\alpha\},\\
    \frac{du}{dx}(\alpha^-),& x= \alpha.
    \end{cases}
    \label{eqn:definition_of_derivative}
\end{equation}
}
By dividing $I$ into $N$ sub-intervals, we obtain the following partition of $I$:

$$I_k=(x_{k-1},x_k),\quad \T_h=\{I_k\}_{k=1}^N,\quad a=x_0<x_1<\dots<x_N=b,\qquad h=\max_{1\le k\le N}(x_{k}-x_{k-1}).$$
We will assume that there is $k_0\in\{1,2,\dots,N\}$ such that $x_{k_0-1}<\alpha<x_{k_0}$, which is equivalent to $\alpha\in \overset{\circ}{I}_{k_0}$. We define the discontinuous immersed finite element space $W^m_{\alpha,r}$ on the interval $I$ as

\begin{equation}{\cal W}^m_{\alpha,r}(\T_h)=\left\{\p\mid \p_{\mid{I_k}}\in\P^m(I_k) \text{ for } k\in\{1,\dots,N\}\backslash\{k_0\} \text{ and } \p_{\mid I_{k_0}}\in \V^m_{\alpha,r}(\Ik)\right\},\label{eqn:global_IFE_space}\end{equation}
where $\P^m(I_k)$ is the space of polynomials of degree at most $m$ on $I_k$ and $\V^m_{\alpha,r}(I_{k_0})$ is the local immersed finite element (\textit{LIFE}) space defined as:
\begin{align}
\V^m_{\alpha,r}(\Ik)= \left\{\varphi \in \C^{m}_{\alpha,r}(\Ik) \mid \p_{\mid \Ik^s}\in \P^m\left(\Ik^s\right),\ s=+,-\right\}, ~0 \leq m \leq {\tilde m}.
\end{align}

In discussions from now on, given a function $v$ in $\H^{m+1}_{\alpha,r}$ (or $\cC^{m}_{\alpha,r}$ or $\V^m_{\alpha,r}$), its derivative is understood in the piecewise sense unless specified otherwise. By definition, we can readily verify that $\V^{m-1}_{\alpha,r}(\Ik)\subset\V^{m}_{\alpha,r}(\Ik)$ for
a given finite sequence $r=(r_i)_{i=0}^{\tilde m}, ~{\tilde m} \geq 1$.
%We note that the space $\V^m_{\alpha,r}(\Ik)$ is a generalization of the LIFE spaces used in %\cite{adjeridPthDegreeImmersed2009, heApproximationCapabilityBilinear2008,heImmersedFiniteElement2011} %with $\V^{m-1}_{\alpha,r}(\Ik)\subset\V^{m}_{\alpha,r}(\Ik)$.
In order to  study the LIFE space $\V^m_{\alpha,r}(\Ik)$, we will investigate the properties of the
corresponding reference IFE (\textit{RIFE}) space $\V^m_{\calpha,r}(\cI)$ on the reference interval $\cI=[0,1]$ with an interface point $\calpha\in(0,1)$. Our goal is to extend the scaling argument to
such IFE spaces and use the IFE scaling argument to show IFE spaces such as $\V^m_{\alpha,r}(\Ik)$
have the optimal approximation capability, i.e., every function in $\H^{m+1}_{\alpha,r}(\Ik)$ can be approximated by functions from the IFE space $\V^m_{\alpha,r}(\Ik)$ at the  optimal convergence rate.

Following the convention in the error analysis literature for finite element methods, we will often use a generic constant $C$ in estimates whose value varies depending on the context, but this generic constant is independent of $h$ and the interface $\alpha \in I_{k_0}$ or $\calpha \in \cI$ unless otherwise declared.

\section{Properties of the RIFE space \texorpdfstring{$\cV$}{VI}}
\label{sec:properties_of_v}

\commentout{
        In this section, we will present some properties of the RIFE space $\V^m_{\calpha,r}(\cI)$ that will be used to establish the approximation capability of the space and the existence and approximation capability of the immersed Radau projection that will be a key tool for proving  the optimal convergence rate of the immersed DG scheme for the kinematic wave equation with discontinuous coefficients as well as the 1-D acoustic wave equation with discontinuous coefficients.
}
For a given function $\cp\in\V^m_{\calpha,r}(\cI)$, we will write $\cp=(\cp_-,\cp_+)$ where $\cp_{s}=\cp_{\mid\cI^s}\in \P^m(\cI^s)$ for $s=+,-$.   Additionally, we will use $\cp_{s}^{(k)}(\calpha)$ to denote $\lim_{x\to \calpha^s}\cp_s^{(k)}(x), ~s = \pm$ for a given integer $k\ge 0$. For clarity, we will use $s'$ to denote the dual of $s$, i.e., if $s=\pm$, then $s'=\mp$.

\begin{lemma}%[Dimension and uniqueness of extensions]
\label{lem:Dimension and uniqueness of extensions} Let ${\tilde m} \geq m\ge 0, \{r_k\}_{k=0}^{\tilde m} \subset\mathbb{R}_+, \calpha\in (0,1)$ and $s\in\{+,-\}$. The following statements hold
\begin{enumerate}
    \item For every $\cp_{s}\in \P^m(\cI^s)$ there is a unique $\cp_{s'} \in \P^m(\cI^{s'})$ such that $\cp=(\cp_-,\cp_+)\in\cV$.

 \item The dimension of $\V^m_{\calpha,r}(\cI)$ is $m+1$.

 \item The set $\{\N^k_{\calpha,r}\}_{k=0}^m$, where
   \begin{equation}\N^k_{\calpha,r}(x)=\begin{cases}
    (x-\calpha)^k,& x\in \cI^-,\\
    r_k(x-\calpha)^k,& x\in \cI^+,
    \end{cases}\label{eqn:definition_of_canonical_basis}\end{equation}
    forms a basis of $\V^m_{\calpha,r}(\cI)$ and will be referred to as the canonical basis.
    \end{enumerate}
\end{lemma}

\begin{proof} We will prove the statements in order:
\begin{enumerate}
\item Let $\cp_{\pm}\in \P^m(\cI^{\pm})$, then  $(\cp_-,\cp_+)\in \cV$ if and only if $$\cp_{\mp}^{(k)}(\calpha)= \left(r_k\right)^{\mp 1} \cp_{\pm}^{(k)}(\calpha),\qquad k=0,1,\dots,m,$$
 which uniquely defines a polynomial $\cp_{\mp}\in \P^m(\cI^\mp)$: $$\cp_{\mp}(x)=\sum_{k=0}^m \frac{\left(r_k\right)^{\mp 1} \cp_{\pm}^{(k)}(\calpha)}{k!} (x-\calpha)^k.$$

\item We have shown that the maps $\cp_-\mapsto \cp$ is well defined and injective which implies that the map $\cp\mapsto \cp_-$ is surjective since every $\cp_-\in \P^m(\cI^-)$ can be extended to $\cp\in\cV$. Hence, $\cV$ is isomorphic to $\P^m(\cI^-)$ implying that the dimension of $\V^m_{\calpha,r}(\cI)$ is $m+1$.

\item We only need to show that $\{\N^k_{\calpha,r}\}_{k=0}^m$ is linearly independent: Assume that $\cp=\sum_{k=0}^mc_k\N^k_{\calpha,r}\equiv 0$, then $\cp_-\equiv 0$ which implies that $c_k=0$ for all $k=0,1,\dots,m$.
\end{enumerate}
\end{proof}

The results in \autoref{lem:Dimension and uniqueness of extensions} allows us to introduce
an extension operator that maps $\cp_s$ to $\cp_{s'}$.

\begin{definition}
Let ${\tilde m} \geq m\ge 0, \{r_k\}_{k=0}^{\tilde m} \subset\mathbb{R}_+, \calpha\in (0,1)$ and $s\in\{+,-\}$. We define the extension operator $\Esp:\P^m(\cI^s)\to\P^m(\cI^{s'})$ that maps every $\cp_{s}\in \P^m(\cI^s)$ to $\Esp(\cp_s)=\cp_{s'}$ such that $(\cp_{-},\cp_+)\in\cV$.
\end{definition}

By \autoref{lem:Dimension and uniqueness of extensions}, the extension operator $\Esp$ is well-defined and is linear. Furthermore, by \autoref{lem:Dimension and uniqueness of extensions} again, this extension operator is also invertible. Consequently, the dimension of the RIFE space is the same as the dimension of the traditional polynomial space of the same degree.

% \comment{I kept the definition of $s,s'$ here (instead of section 2) since they do not show up much in the later sections}

% In the following lemma, we show that the RIFE space is isomorphic to $\P^m(\cI^s)$ for $s=+,-$ and we define a canonical basis.

Next, we will estimate the operator norm of $\Esp$. Let $\ch_-=\calpha,\ \ch_+=1-\calpha$ be the lengths of the sub-intervals $\cI^\pm$ formed by $\calpha$. First, let us consider the following example $\cp=(\cp_-,\cp_+)=\N^m_{\calpha,r}$ defined by \eqref{eqn:definition_of_canonical_basis}, we have

$$\norm{\Esp(\cp_-)}_{0,\cI^+}=\norm{\cp_+}_{0,\cI^+}= |r_m|\left(\frac{\ch_+}{\ch_-}\right)^{m}\sqrt{\frac{\ch_+}{\ch_-}}\norm{\cp_-}_{0,\cI^-},\qquad \text{where } \ch_-=\calpha,\ \ch_+=1-\calpha.$$
Hence, if $h_->h_+$, we get $\norm{\cp_+}_{0,\cI^+}\le |r_m|\norm{\cp_-}_{0,\cI^-}$. In the following lemma, we will show that a similar result holds for all $\cp\in \cV$. Consequently, for every interface position $\calpha \in I$, one of the two extension operators $\Esp, ~s' = -, +$ will be bounded independently of $\calpha$.

\begin{lemma}%[Norms of the extension]
\label{lem:norm_of_extension}
There exists a constant $C>0$ that depends on $m$ such that for every $(\cp_-,\cp_+)\in \cV$, we have
\begin{equation}\norm{\Esp\cp_s}_{0,\cI^{s'}}=\norm{\cp_{s'}}_{0,\cI^{s'}}\le C \sqrt{\frac{\ch_{s'}}{\ch_{s}}}
    \left(\max_{0\le i\le m}r_i^{{\tilde s}}\right)\max\left(1,\left(\frac{\ch_{s'}}{\ch_{s}}\right)^m\right)\norm{\cp_{s}}_{0,\cI^{s}},\qquad s=+,-,\label{eqn:raw_norm_Emp}
\end{equation}
where ${\tilde s} = {\color{red}\mp}1$ for $s = \pm$.
\comment{
    Notice that I wrote $r_i^{{\tilde s}}$ which means $r_{i}^{\pm 1}$.
}
In particular, if $\ch_s\ge \ch_{s'}$, we have
     \begin{equation}
    \norm{\cp_{s'}}_{0,\cI^{s'}}\le C \sqrt{\ch_{s'}}
    \left(\max_{0\le i\le m}r_i^{{\tilde s}}\right)\norm{\cp_{s}}_{0,\cI^s}
    .\label{eqn:norm_of_smallest}\end{equation}

\end{lemma}

\begin{proof}
  First, we note that \eqref{eqn:norm_of_smallest} is a straightforward consequence of \eqref{eqn:raw_norm_Emp}. Here, we only need prove \eqref{eqn:raw_norm_Emp} for $s=-$ since the case $s=+$ can be proven similarly. For every $(\cp_-,\cp_+)\in \cV$, we first define $\hpm\in \P^m([0,1])$ as $\hpm(\xi)=\cp_-(\ch_- \xi)$ which yields

  \begin{equation}\hpm^{(i)}(1)= \ch_-^{i}\cp_-^{(i)}(\calpha),\quad i=0,1,\dots,m.
  \label{eqn:phi_ref_map}
  \end{equation}
Now, let us write $\cp_+$ as a finite Taylor sum around $\calpha$ and use $\cp_+^{(i)}(\calpha)=r_i\cp_-^{(i)}(\calpha)$ to obtain:
   $$\cp_+(x)=\sum_{i=0}^m \cp_+^{(i)}(\calpha)\frac{(x-\calpha)^i}{i!}
   =\sum_{i=0}^m r_i\cp_-^{(i)}(\calpha)\frac{(x-\calpha)^i}{i!}.
$$
Using \eqref{eqn:phi_ref_map}, we can replace $\cp_-^{(i)}(\calpha)$ by $\ch_-^{-i}\hpm^{(i)}(1)$:
\begin{equation}
\cp_+(x)=\sum_{i=0}^m r_i \hpm^{(i)}(1)\ch_-^{-i}\frac{(x-\calpha)^i}{i!}. \label{eqn:final_p+}
\end{equation}
We square and integrate \eqref{eqn:final_p+}, then we apply the change of variables $z=x-\calpha$ to get
$$\norm{\cp_+}_{0,\cI^+}^2= \int_{\calpha}^1\left(\sum_{i=0}^m r_i \hpm^{(i)}(1)\ch_-^{-i}\frac{(x-\calpha)^i}{i!} \right)^2\ dx= \int_{0}^{\ch_+}\left(\sum_{i=0}^m r_i \hpm^{(i)}(1)\left(\frac{z}{\ch_-}\right)^i \frac{1}{i!}\right)^2\ dz.$$
We can bound $r_i$ and $|\hpm^{(i)}(1)|$ by their maximum values for $0\le i\le m$ and we can bound $\left(\frac{z}{\ch_-}\right)^i $ by $$\left(\frac{z}{\ch_-}\right)^i \le \max\left(1,\left(\frac{\ch_+}{\ch_-}\right)^m\right).$$
We also have $\sum_{i=0}^m \frac{1}{i!}\le e$.  Using these  bounds, we get

\begin{align}\norm{\cp_+}_{0,\cI^+}^2&\le \left(\max_{0\le i\le m}r_i\right)^2
\left(\max_{0\le i\le m}
|\hpm^{(i)}(1)|\right)^2\max\left(1,\left(\frac{\ch_+}{\ch_-}\right)^m\right)^2\int_0^{\ch_+} e^{2}\ dz\notag\\ & = \left(\max_{0\le i\le m}r_i\right)^2
\left(\max_{0\le i\le m}|\hpm^{(i)}(1)|\right)^2\max\left(1,\left(\frac{\ch_+}{\ch_-}\right)^m\right)^2 \ch_+e^2.\label{eqn:first_p+_bound}\end{align}
Since $\P^m([0,1])$ is a finite dimensional space, all norms are equivalent. In particular, there is a constant $C(m)$ such that
\begin{equation*}
\left(\max_{0\le i\le m}|p^{(i)}(1)| \right)\le C(m)\norm{p}_{0,[0,1]}, ~\forall p\in \P^m([0,1]), %\label{eqn:equivalence_of_norms}
\end{equation*}
which leads to
\begin{equation}\left(\max_{0\le i\le m}|\hpm^{(i)}(1)|\right)^2\le C(m)\norm{\hpm}_{0,[0,1]}^2.\label{eqn:equivalence_of_norms_revisited}\end{equation}
By using a change of variables, we can show that \begin{equation}\norm{\hpm}_{0,[0,1]}^2=\frac{1}{\ch_-}\norm{\cp_-}_{0,\cI^-}^2.\label{eqn:L2_scaling}\end{equation}
Finally, we combine \eqref{eqn:first_p+_bound}, \eqref{eqn:equivalence_of_norms_revisited} and \eqref{eqn:L2_scaling} to get

$$\norm{\mathcal{E}_{\check{\alpha}, r}^{m, +}\cp_-}^2 =  \norm{\cp_+}_{0,\cI^+}^2\le C(m)\left(\max_{0\le i\le m}r_i\right)^2\left(\frac{\ch_+}{\ch_-}\right)\max\left(1,\left(\frac{\ch_+}{\ch_-}\right)^m\right)^2\norm{\cp_-}_{0,\cI^-}^2$$
which is \eqref{eqn:raw_norm_Emp} for $s = -$.

\end{proof}

Next, we will use the bounds on the extension operator $\Es$ to establish inverse inequalities which are independent of $\calpha$ for the RIFE space.

\commentout{
        \begin{lemma}\label{lemma:inverse_inequality} %[The inverse inequality on $\cI$]
        Let ${\tilde m} \geq m\ge 0, \{r_k\}_{k=0}^{\tilde m} \subset\mathbb{R}_+$, and $ \calpha\in (0,1)$, then there is $C(m,r)>0$ that depends on the sequence $r$ and the degree $m$ such that  for every $\cp\in \V^m_{\calpha,r}(\cI)$, we have
        \begin{align}
        \norm{\cp'}_{0,\cI}\le C(m,r) \norm{\cp}_{0,\cI}, ~~(\text{{\color{red}for higher order derivatives???}}) \label{eqn:inverse_inequality}
        \end{align}
        where $\cp'=(\cp_-',\cp_+')$.
        \end{lemma}
}

\begin{lemma} \label{lemma:inverse_inequality}
Let $\tilde{m}\ge m\ge 0, \{r_k\}_{k=0}^{\tilde{m}}\subset\mathbb{R}_+$,  $ \calpha\in (0,1)$. Then there exists $C(m,r)>0$ independent of $\calpha$ such that for every $\cp\in \V^m_{\calpha,r}(\cI)$ we have
\begin{align}
|\cp|_{i,\cI}\le C(m,r) \norm{\cp}_{0,\cI}, ~0 \leq i \leq m+1. \label{eqn:inverse_inequality}
\end{align}
\end{lemma}

\begin{proof}
The estimate given in \eqref{eqn:inverse_inequality} obviously holds for $i = 0$ and $i = m+1$. Without loss of generality, assume that $\ch_-\ge \ch_+$, this implies that $\ch_-\ge \frac{1}{2}$. Then, using the classical inverse inequality \cite{brennerPolynomialApproximationTheory1994} we have

\begin{equation}\norm{\cp'_-}_{0,\cI^-}\le C \ch_-^{-1} \norm{\cp_-}_{0,\cI^-}\le 2C \norm{\cp_-}_{0,\cI^-}.\label{eqn:first_inv}
\end{equation}
By the Taylor expansion of $\cp'_+(x)$ at $x = \check{\alpha}$, we have
\begin{equation}
\cp'_+ = \mathcal{E}^{m-1,+}_{\calpha, \tau(r)}\left(\ddx \cp_-\right), \label{eqn:first_shift_operator}
\end{equation}
where $\tau:(r_0,r_1,\dots,r_m)\mapsto (r_1,\dots,r_m)$ is the shift operator. By \eqref{eqn:norm_of_smallest} and the inverse inequality given in \eqref{eqn:first_inv}, we have
\begin{equation}
\norm{\cp'_+}_{0,\cI^+}\le C(m)\left(\max_{1\le i \le m}r_i\right) \norm{\cp_-'}_{0,\cI^-}\le C(m,r) \norm{\cp_-}_{0,\cI^-}. \label{eqn:second_inv}
\end{equation}
Therefore, we have
$$\abs{\cp}_{1, \cI} = \norm{\cp'}_{0,\cI}\le C(m,r) \norm{\cp_-}_{0,\cI^-}\le C(m,r) \norm{\cp}_{0,\cI}$$
which proves \eqref{eqn:inverse_inequality} for $i=1$.
Applying similar arguments, we can prove \eqref{eqn:inverse_inequality} for other values of $i$.

\comment{Here's a proof of \eqref{eqn:first_shift_operator}: Let $$\cp_-(x)=\sum_{i=0}^m c_i(x-\calpha)^i,$$
then $\cp_+=\Ep(\cp_-)$ is
$$\cp_+(x)=\sum_{i=0}^m r_i c_i(x-\calpha)^i.$$
By taking the derivative and rearranging the sum, we get
$$\cp_+(x)=\sum_{i=0}^{m-1} r_{\color{red}i+1} (i+1)c_{i+1}(x-\calpha)^i,\qquad \cp_-'(x)=\sum_{i=0}^{m-1}  (i+1)c_{i+1}(x-\calpha)^i.$$

Hence, $\cp'_+= \mathcal{E}^{m-1,+}_{\calpha, \tau(r)}\ddx \cp_-$.
}

\end{proof}

Since $\cp_+= \Ep \left(\cp_-\right)$, the formula in \eqref{eqn:first_shift_operator} leads to the following identity about the permutation of the classical differential operator and the extension operator:
\begin{equation}
\ddx \Ep \left(\cp_-\right)= \cp'_+ = \mathcal{E}^{m-1,+}_{\calpha, \tau(r)}\left(\ddx \cp_-\right),
~\forall m \geq 1. \label{eqn:DiffOp_EOp_exchange}
\end{equation}

As a piecewise function $\cp = (\cp_-, \cp_+) \in \cV$, the value of $\cp$ at $\calpha$ is not defined in general since the two sided limits $\cp_-(\calpha)$ and $\cp_+(\calpha)$ could be different if $r_0\ne 1$. However, if $\cp_s(\calpha)=0$ then $\cp_{s'}(\calpha)=0$ for $s=+,-$. Furthermore, the multiplicity of $\calpha$ as a root of $\cp_-$ is the same as its multiplicity as a root of $\cp_+$. This observation motivates
to define  $\calpha$ as a root of $\cp$ of multiplicity $d$ if $\calpha$ is a root of $\cp_-$ of multiplicity $d$. The following theorem shows that the number of roots of a non-zero function $\cp\in \cV$ counting multiplicities cannot exceed $m$ (similar to a polynomial of degree $m$), this theorem will be crucial to establish the existence of a Lagrange-type basis in $\cV$ and constructing an immersed Radau projection later in \autoref{sec:one_D_class}.

\comment{There is an unproven claim here: if $\calpha$ is a roof of $\cp_-$ of multiplicity $d$, then it is a root of $\cp_+$ of the same multiplicity but it is easy to prove:

We have $\cp_-^{(k)}(\calpha)=0$ for $k=0,1,\dots d-1.$ and $\cp_-^{(d)}(\calpha)=0$, then the same holds for $\cp_+$.
}

In the discussions below, we will omit the phrase ``counting multiplicities" for the sake of conciseness. For example, we say that $(x-2)^2$ has two roots in $\mathbb{R}$.

\begin{theorem}%[Zeros of IFE basis functions]
\label{thm:Zeros of IFE basis functions}
For ${\tilde m} \geq m\ge 0, \{r_k\}_{k=0}^{\tilde m} \subset\mathbb{R}_+$, and $ \calpha\in (0,1)$, every non-zero $\cp \in \V^m_{\calpha,r}(\cI)$ has at most $m$ roots.
\end{theorem}

\begin{proof} We start from the base case $m=0$ and then proceed by induction. Let $\cp\in\V^0_{\calpha,r}(\cI)$. if $\cp\not\equiv0$ then $\cp=c\N^0_{\calpha,r}$ for some $c\ne 0$. In this case, $\cp$ has no roots since $\N^0_{\calpha,r}$ has no roots.
 Now, assume that for every positive sequence $q$, the number of roots of any non-zero $\cp\in \V^{m-1}_{\calpha,q}(\cI)$ is at most $m-1$. Next, we will show that for a given positive sequence $r$, every function $\cp\in \cV$ has at most $m$ roots by contradiction:

 Assume that $\cp = (\cp_-, \cp_+) \in \V^{m}_{\calpha,r}(\cI)$ is a non-zero function that has $j$ disctinct roots $\{\xi_i\}_{i=1}^j$ of multiplicities $\{d_i\}_{i=1}^j$ such that $D=d_1+d_2+\dots+d_j>m$. Therefore, $\xi_j>\calpha$ and $\xi_1\le \calpha$ because $\cp_{\pm}\in\P^m(\cI^s)$. let $\xi_{i_0}$ be the largest root
 that is not larger than $\calpha$, i.e.,
 $$0\le \xi_1<\xi_2<\dots<\xi_{i_0}\le \calpha<\xi_{i_0+1}<\dots<\xi_j\le 1,$$

%  \comment{If a polynomial $p$ has $D$ roots in an interval $[a,b]$, then $p'$ has $D-1$ roots in $[a,b]$.}

 By the definition of $\cp = (\cp_-, \cp_+)$,  $\cp_-$ has $D_1=d_1+d_2+\dots+d_{i_0}$ roots in $[\xi_{1},\xi_{i_0}]$ and $\cp_+$ has $D_2=d_{i_0+1}+d_{i_0+2}+\dots+d_j$ roots in $[\xi_{i_0+1},\xi_j]$. Therefore, $\cp_-'$ has $D_1-1$ roots in $[\xi_{1},\xi_{i_0}]$ and $\cp_+'$ has $D_2-1$ roots in $[\xi_{i_0+1},\xi_j]$. It remains to show that $\cp'$ has an additional root in $(\xi_{i_0},\xi_{i_0+1})$. To show that,we consider two cases:
 \begin{itemize}
     \item $\xi_{i_0}=\calpha$: In this case, $\cp$ is continuous and $\cp_+(\calpha)=\cp_+(\xi_{i_0+1})=0$. By the mean value theorem, we conclude that $\cp_+'$ has a root in $(\xi_{i_0},\xi_{i_0+1})$.

     \item $\xi_{i_0}<\calpha$: Assume that $\cp'(x)> 0$ for all $x\in(\xi_{i_0},\xi_{i_0+1})\backslash\{\calpha\}$, then $\cp_-(\calpha)>0$. Since $r_0>0$, we have $\cp_+(\calpha)>0$. By integrating $\cp_+'$ from $\calpha$ to $\xi_{i_0+1}$, we get $0=\cp_{+}(\xi_{i_0+1})>0$, a contradiction. A similar conclusion follows if we assume that $\cp'(x)< 0$ for all $x\in(\xi_{i_0},\xi_{i_0+1})\backslash\{\calpha\}$. Therefore, $\cp'$ changes sign at some  $x_0\in (\xi_{i_0},\xi_{i_0+1})$. Since, $\cp'$ does not change sign at $\calpha$, then $x_0\ne \calpha$ and $\cp'(x_0)=0$ (because $\cp'$ is continuous at $x_0$).
 \end{itemize}

 In either case above, $\cp'$ has $(D_1-1)+(D_2-1)+1=D-1>m-1$ roots in $\cI$ which contradicts the induction hypothesis since $\cp'\in \V^{m-1}_{\calpha,\tau(r)}(\cI)$. Therefore, $\cp$ has at most $m$ roots.

%  There are two cases:
% \begin{itemize}
%     \item $\xi_{i_0}=\calpha$: In this case, $\cp_-$ has $D_1=d_1+d_2+\dots d_{i_0}$ roots in $[0,\calpha]$. Therefore, $\cp_-$ has $D_1-1$ roots in $[0,\calpha]$. On the other hand, $\cp_+$ has $D_2=d_{i_0+1}+d_{i_0+2}+\dots+d_j$ roots in $[\xi_{i_0+1},1]$ which implies that $\cp_{+}'$ has $D_2-1$ roots in $[\xi_{i_0}+1]$. By MVT, $\cp_{+}'$ has yet another root in $(\calpha,x_{i_0+1})$. Therefore, $\cp'$ has $(D_1-1)+(D_2-1)+1=D-1>m-1$ roots in $\cI$ which contradicts the induction hypothesis since $\cp'\in \V^{m-1}_{\calpha,\tau(r)}(\cI)$.

% \end{itemize}

\end{proof}

The previous theorem allows us to establish the existence of a Lagrange basis on $\cV$ for every choice of nodes and for every degree $m$ which was proved by Moon in \cite{moonImmersedDiscontinuousGalerkin2016} for a few specific cases $m=1,2,3,4$.

\begin{theorem}%[Lagrange basis]
\label{thm:Lagrange basis}
 Let ${\tilde m} \geq m\ge 0, \{r_k\}_{k=0}^{\tilde m} \subset\mathbb{R}_+$, and $ \calpha\in (0,1)$. Assume $\xi_0,\xi_1,\dots,\xi_{m}$ are $m+1$ distinct points in $\cI$, then there is a Lagrange basis $\{L_i\}_{i=0}^{m}$ of $\V^m_{\calpha,r}(\cI)$ that satisfies
\begin{align}
L_i(\xi_j)=\delta_{i,j},\qquad \  0\le i,j\le m. \label{eq:Lagrange basis}
\end{align}
\end{theorem}

\begin{proof} For each $0\le i\le m$, we construct $\tilde{L}_i\in \cV$ such that $\tilde{L}_{i}(\xi_j)=0$ for all $j\ne i$ by writing $\tilde{L}_i$ as
$$\tilde{L}_i=\sum_{i=0}^{m}a_i\N^i_{\calpha,r},$$
for some $\{a_i\}_{i=0}^m$ chosen such that \begin{equation}\tilde{L}_{i}(\xi_j)=\sum_{i=0}^{m}a_i\N^i_{\calpha,r}(\xi_j)=0,\quad \forall j\in\{0,1,\dots,i-1,i+1,\dots,m\}.
\label{eqn:system_for_Lagrange}
\end{equation}
The equations \eqref{eqn:system_for_Lagrange} form a homogeneous system of $m$ equations with $m+1$ unknowns. Therefore, it has a non-zero solution.
From  \autoref{thm:Zeros of IFE basis functions}, we know that $\tilde{L}_{i}(\xi_i)\ne 0$; otherwise, $\tilde{L}_i$ would have $m+1$ roots. This allows us to define $L_i$ as
$$L_i(x)=\frac{1}{\tilde{L}_i(\xi_i)}\tilde{L}(x).$$

By \eqref{eq:Lagrange basis}, $L_i, 0 \leq i \leq m$ are linearly independent. Consequently, $\{L_i\}_{i=0}^m$ is a basis for $\cV$ since its dimension is $m+1$ from \autoref{lem:Dimension and uniqueness of extensions}.

\end{proof}

In addition to having a Lagrange basis, the RIFE space has an orthogonal basis with respect to $(\cdot,\cdot)_{w,\cI}$ as stated in the following theorem in which we also show that if a function $\cp\in \cV$ is orthogonal to $\V^{m-1}_{\calpha,r}(\cI)$ with respect to $(\cdot,\cdot)_{w,\cI}$, then $\cp$ has exactly $m$ distinct interior roots similar to the classical orthogonal polynomials. Although the theorem holds for  a general weight $w$, we will restrict our attention to a piecewise constant function $w$:
\begin{equation}w(x)=\begin{cases}
w_-,&x\in \cI^-,\\ w_+,& x\in \cI^+,
\end{cases}
\label{eqn:definition_of_w}
\end{equation}
where $w_{\pm}$ are positive constants.  The result of this theorem can also be considered as a generalization for the theorem about the orthogonal IFE basis described in \cite{caoSuperconvergenceImmersedFinite2017} for elliptic interface problems.

\begin{theorem}%[Zeros of orthogonal IFE functions]
\label{thm:Zeros of orthogonal IFE functions}
 Let ${\tilde m} \geq m\ge 1, \{r_k\}_{k=0}^{\tilde m} \subset\mathbb{R}_+$, $ \calpha\in (0,1)$, and let $w:\cI\to \mathbb{R}_+$ be defined as in \eqref{eqn:definition_of_w}, then there is a non-zero $\cp\in \cV$ such that
\begin{align}
(\cp,\cpsi)_{w,\cI} =\int_{\cI} w(x)\cp(x)\cpsi(x)\ dx= 0,\quad \forall \cpsi\in \V^{m-1}_{\calpha, r}(\cI). \label{eq:Zeros of orthogonal IFE functions}
\end{align}
Furthermore, $\cp$ has exactly $m$ distinct roots in the interior of $\cI$.
\end{theorem}

\begin{proof}
 Existence is a classical result of linear algebra.
The proof of the second claim follows the same steps used for orthogonal polynomials:  Note that $\cp$ has at least one root of odd multiplicity in the interior of $\cI$ since $(\cp,\N^0_{\calpha,r})_{w,\cI}=0$.

Assume that $\cp$ has $j<m$ distinct roots $\{\xi_i\}_{i=1}^j$ of odd multiplicity in the interior of $\cI$. Following the ideas in the proof of \autoref{thm:Lagrange basis}, we can show that there is $\cpsi_0\in\V^{j}_{\calpha,r}(\cI)$ such that $\cpsi_{0}(\xi_i)=0$ for $1\le i\le j$.

Furthermore, all roots of $\cpsi_0$ are simple according to \autoref{thm:Zeros of IFE basis functions} since the sum of multiplicities cannot exceed $j$, and $\cpsi_0$ changes sign at these roots. This means that $w\cp\cpsi_0$ does not change sign on $\cI$. As a consequence, $(\cp,\cpsi_0)_{w,\cI}\ne 0$ which contradicts the assumption $(\cp,\cpsi)_{w,\cI}=0$ for all $\cpsi\in \V^{m-1}_{\calpha,r}(\cI)$ since $\V^{j}_{\calpha,r}(\cI)\subseteq\V^{m-1}_{\calpha,r}(\cI)$.
\end{proof}

For every integer $m$ with ${\tilde m} \geq m \geq 1$, we use $\cQ, ~m \geq 1$ to denote the orthogonal complement of $\V^{m-1}_{\calpha,r}(\cI)$ in $\cV$ with respect to the weight $w$, that is
$$\cQ=\left\{\cp\in \cV\mid (\cp,\cpsi)_{w,\cI}=0,\ \forall\cpsi\in \V^{m-1}_{\calpha,r}(\cI)\right\}.$$
According to \autoref{thm:Zeros of orthogonal IFE functions}, one can see that  $\cp\mapsto \sqrt{\cp(0)^2+\cp(1)^2}$ defines a norm on $\cQ$ which is one-dimensional. Thus, it is is equivalent to the $L^2$ norm and the quantity $\frac{\sqrt{\cp(0)^2+\cp(1)^2}}{\norm{\cp}_{0,\cI}}$ depends only on $\calpha,w$ and $r$ (and not on the choice of $\cp\in\cQ$). Furthermore, The following lemma shows that the equivalence constant is independent of the interface location. This result will be crucial later in the analysis of Radau projections.

\begin{lemma}
\label{lem:norm_on_Q}
Let ${\tilde m} \geq m\ge 1, \{r_k\}_{k=0}^{\tilde m} \subset\mathbb{R}_+$, $ \calpha\in (0,1)$ and $w$ be defined as in \eqref{eqn:definition_of_w} then, there exist $C(m,w,r)$ and $\tilde{C}(m,w,r)>0$ independent of $\calpha$ such that for every $\cp\in \cQ$, we have
\begin{align}
\sqrt{\cp(0)^2+\cp(1)^2} \ge C(m,w,r)\norm{\cp}_{0,\cI}\ge\tilde{C}(m,w,r)\norm{\cp}_{1,\cI}.
\label{eq:norm_on_Q}
\end{align}
\end{lemma}

\begin{proof} The inequality on the right follows from the inverse inequality \eqref{eqn:inverse_inequality} for the IFE funcdtions. For a proof of the inequality on the left, see \autoref{sec:proof_of_lemma_4}.

\end{proof}

\section{An immersed Bramble-Hilbert lemma and the approximation capabilities of the LIFE space}
\label{sec:Bramble_Hilbert}
In this section, we will develop a new version of Bramble-Hilbert lemma that applies to functions in  $\H^{m+1}_{\calpha,r}(\cI)$ and its IFE counterpart. This lemma will serve as a fundamental tool for investigating the approximation capability of IFE spaces. In the discussions below, we will use $\mathbbm{1}_{B}$ for the indicator function of a set $B\subset\mathbb{R}$ and we define $w_i=r_i\mathbbm{1}_{\cI^-}+\mathbbm{1}_{\cI^+}$ for $i=0,1,\dots,m$.

% \comment{
% Are the spaces $\H_{\calpha,r}^{m}(\cI)$ and  $\X^m_{\calpha,r}(\cI)$ essentially the same? It seems like it.
% }

%First we note that for every function $f\in \C^m{\calpha,r}(\cI)$, we have
%$$\ddx f \in \C^{m-1}_{\calpha,\tau(r)}(\cI),$$
%where $\tau$ is the shift operator mentioned in \eqref{eqn:first_shift_operator}.

\begin{theorem}\label{thm:immersed_bramble_hilbert}
 Let ${\tilde m} \geq m\ge 0, \{r_k\}_{k=0}^{\tilde m} \subset\mathbb{R}_+$, $\calpha\in(0,1)$, and $v\in \H^{m+1}_{\calpha,r}(\cI)$.
 Assume $(w_i,v^{(i)})_{\cI}=0$ for $i=0,1,\dots, m$. Then, there exists $C(i,r)>0$ independent of $\calpha$ such that
\begin{align}
\norm{v}_{i,\cI}\le C(i,r)|v|_{i,\cI},~i = 0, 1, \cdots, m+1. \label{eq:immersed_bramble_hilbert_1}
\end{align}
\end{theorem}

\begin{proof}
Let $v\in \H^{m+1}_{\calpha,r}(\cI)$ and assume $(w_i,v^{(i)})_{\cI}=0$ for $i=0,1,\dots, m$.
Because $w_0$ is such that $w_0v$ is continuous and since $v_{\mid \cI^{\pm}}\in H^1(\cI^{\pm})$, we have $w_0v\in H^1(\cI)$. Therefore, for any given $x,y\in\cI$, we have
$$w_0(x)v(x)-w_0(y)u(y)=\int_{y}^x w_0(z)v'(z)\ dz$$
We integrate this identity on $\cI$ with respect to $x$ and use $(w_0, v^{(0)})_{\cI} = 0$ to get
$$-w_0(y)v(y)=\int_{0}^1 \int_{y}^x w_0(z)v'(z)\ dz\ dx, \quad \forall y\in \cI.$$
%\begin{equation}0=(w_0,v)_{\cI}= \int_{0}^1 w_0(x)v(x)\ dx \text{~~with~~} w_0(x)=\begin{cases} r_0,& %x\in \cI^-\\ 1,& x\in \cI^+\end{cases}.\label{eqn:w0v}
%\end{equation}
Taking the absolute value and applying the Cauchy-Schwarz inequality, we get
$$|v(y)|\le \frac{1}{\min(1,r_0)} |w_0(y)v(y)|\le\frac{\max(1,r_0)}{\min(1,r_0)}|v|_{1,\cI}=\max\left(r_0,r_0^{-1}\right)|v|_{1,\cI},$$
which implies $\norm{v}_{0, \cI} \leq \max\left(r_0,r_0^{-1}\right)|v|_{1,\cI}$. Since $v'\in \H^{m}_{\calpha,\tau(r)}(\cI)$, where $\tau$ is the shift operator described in \eqref{eqn:first_shift_operator}, we can use the same reasoning to show that
$$|v|_{1,\check{I}} \le \max\left(r_1,r_1^{-1}\right)|v|_{2,\check{I}}.$$
Repeating the same arguments, we can obtain
\begin{align}
|v|_{i,\check{I}} \le \max\left(r_i,r_i^{-1}\right)|v|_{i+1,\check{I}}, ~~i = 0, 1, \cdots, m
\label{eq:immersed_bramble_hilbert_2}
\end{align}
which leads to \eqref{eq:immersed_bramble_hilbert_1} with
$$C(i,r) = \sqrt{1+\sum_{k=0}^{i}\prod_{j=k}^i \max\left(r_j^2,r_j^{-2}\right)}.$$
\end{proof}

\commentout{
        \begin{definition}
             Let $m\ge 0, \{r_k\}_{k=0}^m\subset\mathbb{R}_+$, $\calpha\in(0,1)$ and $v\in \X^m_{\calpha,r}(\cI)$. Let $\cpi v$ be an element of $\cV$ such that
             \begin{equation}\int_{\check I} w_i(x)\frac{d^i}{dx^i}\left(v(x)-\cpi v (x)\right)\ dx=0,\qquad \forall i=0,1,\dots,m.
             \label{eqn:Grossman_orthogonality}
             \end{equation}
        \end{definition}
}

\begin{lemma}\label{lem:hyper_orthogonality}
 Let ${\tilde m} \geq m\ge 0, \{r_k\}_{k=0}^{\tilde m} \subset\mathbb{R}_+$, $\calpha\in(0,1)$ and $v\in \H^{m+1}_{\calpha,r}(\cI)$. Then, there is a unique $\cpi v\in\cV$ that satisfies
\begin{equation}
\int_{\check I} w_i(x)\frac{d^i}{dx^i}\left(v(x)-\cpi v (x)\right)\ dx=0,\qquad \forall i=0,1,\dots,m.     \label{eqn:hyper_orthogonality_1}
\end{equation}
\end{lemma}

\begin{proof}
For $v\in \H^{m+1}_{\calpha,r}(\cI)$, to see that $\cpi v$ exists and is unique, we consider the problem of finding $\cp\in \cV$ such that
\begin{equation}
  (w_i,\cp^{(i)})_{\cI} =  (w_i,v^{(i)})_{\cI},\qquad\text{for } i=0,1,\dots, m.\label{eqn:pi_check_def}
\end{equation}
By \autoref{lem:Dimension and uniqueness of extensions}, we can express $\cp$ in terms of the canonical basis
$$\cp= \sum_{j=0}^m c_j\N^{j}_{\calpha,r}.
$$
Then, by \eqref{eqn:pi_check_def}, the coefficients $\c$ of $\cp$ are determined by the linear system $A\c=\mathbf{b}$, where $A = (A_{i,j})$ is a triangular matrix with $A_{i,j}=(w_i,(\N^j_{\calpha,r})^{(i)})_{\cI}, ~0 \leq i, j \leq m$ and diagonal entries
$$A_{i,i}= (w_i,(\N^i_{\calpha,r})^{(i)})_{\cI} = i!(h_-+r_ih_+) \ne 0, ~0 \leq i \leq m.
$$
Therefore, $A$ is invertible and $\cpi v = \cp$ is uniquely determined by \eqref{eqn:hyper_orthogonality_1}.
\end{proof}

We note that the mapping $\cpi: v\in \H^{m+1}_{\calpha,r}(\cI) \mapsto \cpi v\in\cV$ is linear because
of the linearity of
 integration.

 We now present an immersed version of the Bramble-Hilbert lemma \cite{brambleEstimationLinearFunctionals1970} which can be considered a generalization of the one-dimensional Bramble-Hilbert lemma in the sense that if $r\equiv 1$, then, this immersed Bramble-Hilbert lemma recovers the classical Bramble-Hilbert lemma.

\begin{lemma}\label{lem:projection_error_reference}
Let ${\tilde m} \geq m\ge i\ge 0, \{r_k\}_{k=0}^{\tilde m}\subset\mathbb{R}_+$, $\calpha\in(0,1)$. Assume $\cP: \H^{m+1}_{\calpha,r}(\cI)\to \cV$ is a  linear map that satisfies the following two conditions
\begin{enumerate}
\item $\cP$ is a projection on $\cV$ in the sense that
\begin{equation}
\cP\cp = \cp,\quad \forall \cp\in \cV. \label{eqn:RIFE_preserving}
\end{equation}

\item
There exists an integer $j$,  $0 \leq j \leq m+1$ such that $\cP$ is bounded with respect to the norm $\norm{\cdot}_{j,\cI}$ as follows:
\begin{align}
\norm{\cP v}_{i,\cI}\le C \norm{v}_{j,\cI},\qquad\forall v\in  \H^{m+1}_{\calpha,r}(\cI).
\label{eqn:projection_error_reference_1}
\end{align}

\end{enumerate}
Then, there exists $C(m,r)>0$ independent of $\calpha$, such that
\begin{align}
\norm{v-\cP v}_{i,\cI}\le C(m,r)\left(1+\norm{\cP}_{i, j, \cI}\right)|v|_{m+1,\cI},\quad
v\in \H^{m+1}_{\calpha,r}(\cI), \label{eqn:projection_error_reference_2}
\end{align}
where
\begin{align*}
\norm{\cP}_{i, j, \cI}=\sup\left\{ \norm{\cP v }_{i,\cI}\mid v\in   \H^{m+1}_{\calpha,r}(\cI) \text{ and } \norm{v}_{j,\cI}=1\right\}.
\end{align*}

\end{lemma}

\begin{proof}
Since $\cP$ is a projection in the sense of \eqref{eqn:RIFE_preserving}, we have $\cP\cpi v=\cpi v$. Using the triangle inequality and \eqref{eqn:projection_error_reference_1}, we obtain
\begin{align*}\norm{v-\cP v}_{i,\cI}&\le \norm{v-\cpi v}_{i,\cI}+\norm{\cP\left(v-\cpi v\right)}_{i,\cI}, \\
&\le \left(1+\norm{\cP}_{i, j, \cI}\right) \norm{v-\cpi v}_{j,\cI} \leq
\left(1+\norm{\cP}_{i, j, \cI}\right) \norm{v-\cpi v}_{m+1,\cI}
,\qquad \forall v\in \H^{m+1}_{\calpha,r}(\cI).
\end{align*}
Then, applying \autoref{lem:hyper_orthogonality} and  \autoref{thm:immersed_bramble_hilbert}
to the right hand side of the above estimate leads to \eqref{eqn:projection_error_reference_2}.

\end{proof}

Next, we extend the results of \autoref{lem:projection_error_reference} to the physical interface element $\Ik=[x_{k_0-1},x_{k_0-1}+h]$. Following the tradition in finite element analysis, for every function
$\p$ defined on the interface element $\Ik$, we can map it to a function $\mathcal{M}\p = \cp$ defined on
the reference interval $\cI$ by the standard affine transformation:
\begin{equation}
\mathcal{M}\p(\xi) = \cp(\xi)=\p(x_{k_0-1}+h\xi),\quad \xi\in \cI = [0,1].\label{eqn:translation}
\end{equation}
Furthermore, given a mapping $P^m_{\alpha,r}:\H^{m+1}_{\alpha,r}(\Ik)\to \V^m_{\alpha,r}(\Ik)$, we can use
this affine transformation to introduce a mapping
$\cP: \H^{m+1}_{\calpha,r}(\cI)\to \cV$ such that
\begin{equation}
(\cP \cv)(\xi) = (\Pr v)(x_{k_0-1}+h\xi) = (\Pr v)(x) ~~\text{with}~~
\xi \in \cI ~~\text{or}~~ x = x_{k_0-1}+h\xi \in \Ik. \label{eqn:Projection_LIFE_RIFE}
\end{equation}
It can be verified that the mappings $\mathcal{M}, P^m_{\alpha,r}$ and $\cP$ satisfy the following commutative diagram:
\commentout{
         We first note the LIFE space $\Vk$ is isomorphic to $\cV$ where $h \calpha +x_{k_0-1}=\alpha$ via the standard affine transformation that maps $\p$ to $\cp$ such that

        It is important to note that the sequence $\{r_k\}_{k=1}^{\tilde m}$ that appears in $\cV$ and $\Vk$ is the same since the error estimates obtained earlier depend on $\{r_k\}_{k=1}^m$. To extend the results of \autoref{lem:projection_error_reference} to $\Ik$, we consider a linear bounded operator $\Pr$ defined as

        where $\cv$ is defined similarly to $\cp$ in \eqref{eqn:translation}. We can summarize this construction in the following commutative diagram, where $\mathcal{M}v=\check{v}$.
}
\begin{equation}\includegraphics[valign=c,scale=1.0]{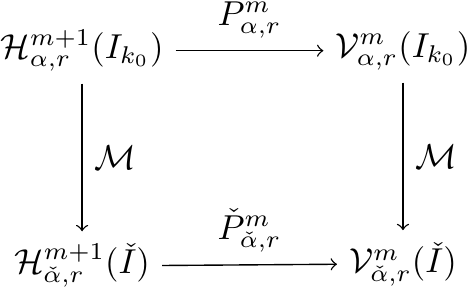}
\label{eqn:phys_ref_diagram}
\end{equation}

We now use the immersed Bramble-Hilbert lemma in the scaling argument to obtain estimates for the  projection error $v-\Pr v$.

\begin{theorem} \label{thm:error_of_general_projection}
 Let ${\tilde m} \geq m\ge i \ge 0, \{r_k\}_{k=0}^{\tilde m} \subset\mathbb{R}_+$, $\calpha\in(0,1)$.
Assume $P^m_{\alpha,r}:\H^{m+1}_{\alpha,r}(\Ik)\to \V^m_{\alpha,r}(\Ik)$ is a linear operator such that
$\cP$ defined by \eqref{eqn:Projection_LIFE_RIFE} satisfies the assumption of
\autoref{lem:projection_error_reference} for an integer $j$ with $0 \leq j \leq m+1$. Then, there exists $C(m,r)>0$ independent of $\alpha$ such that
\begin{equation}
|v-\Pr v|_{i,\Ik}\le Ch^{m+1-i}\left(1+\norm{\cP}_{i, j, \cI}\right)|v|_{m+1,\Ik},\quad
v\in \H_{\alpha,r}^{m+1}(\Ik).\label{eqn:approximation_capabilities_general}
\end{equation}

\end{theorem}

\begin{proof}

The proof follows the same argument as for the classical case $r_k=1,\ k=0,1,\dots,m$ \cite{ciarletFiniteElementMethod2002}. We start by applying the change of variables $x\mapsto h^{-1}(x-x_{k_0-1})$ to we obtain
    \begin{equation}|v|_{m+1,\Ik} = h^{-1}h^{-m-1}|\cv|_{m+1,\cI},\qquad |v-\Pr v|_{i,\Ik} = h^{-1}h^{-i} |\cv-\cP \cv|_{i,\cI}.\label{eqn:random_5}\end{equation}
    Next, we combine \autoref{lem:projection_error_reference} and \eqref{eqn:random_5} to obtain
    \begin{align*}|v-\Pr v|_{i,\Ik}&= h^{-i-1}|\cv-\cP \cv|_{i,\cI}\le h^{-i-1}C(m,r)\left(1+\norm{\cP}_{i, j, \cI}\right)|\cv|_{m+1,\cI}\\
    &=h^{-i-1}C(m,r)\left(1+\norm{\cP}_{i, j, \cI}\right)h^{m+2}|v|_{m+1,\Ik}\\
    &=C(m,r)h^{m+1-i}\left(1+\norm{\cP}_{i, j, \cI}\right)|v|_{m+1,\Ik}.
    \end{align*}

\end{proof}

Nevertheless, the estimate \eqref{eqn:approximation_capabilities_general} does not directly lead to the convergence of $\Pr v$ to $v$ as $h \rightarrow 0$ unless we can show that $\norm{\cP}_{i, j, \cI}$
is uniformly bounded with respect to $\calpha \in \cI$, and this can be addressed by the
\textit{uniform boundedness} of $\Pr$ defined as follows.

\commentout{
        The previous theorem shows that the  norm of $\cP$ plays a role in the error estimates. In particular, the reference interface point $\calpha$ could affect the error estimate, this could be an issue since the relative position of $\alpha$ in $\Ik$ changes under mesh refinement which would change $\calpha$. Fortunately, the projection operators discussed in this paper are bounded by a constant independent of the interface  position relative to the mesh. We shall call such operators \textit{uniformly bounded LIFE/RIFE projections} and we define them below
}

\begin{definition}\label{def:uniform_boundedness}
 Let ${\tilde m} \ge m \ge 0, \{r_k\}_{k=0}^{\tilde m}\subset\mathbb{R}_+$, $\calpha \in \cI$, and
 let $\{\cP\}_{0<\calpha<1}$ be a collection of projections in the sense of
 \eqref{eqn:RIFE_preserving} such that $\cP: \H^{m+1}_{\calpha,r}(\cI)\to \cV$.
 We call $\{\cP\}_{0<\calpha<1}$ a uniformly bounded collection of RIFE projections provided that there exists a constant $C>0$ independent of $\calpha$ and an integer $j$ with $0 \leq j \leq m+1$ such that
\begin{equation}
\norm{\cP v}_{0,\cI}<C\norm{v}_{j,\cI},\ \forall v\in \H^{m+1}_{\calpha,r}(\cI), \qquad \forall \calpha\in(0,1), \label{eqn:def_uniformly_bounded}
\end{equation}
and the associated collection of maps $\{\Pr\}_{\alpha\in\overset{\circ}{I}_{k_0}}$ defined in \eqref{eqn:Projection_LIFE_RIFE} is called a uniformly bounded collection of LIFE projections.
\end{definition}

\begin{lemma}\label{lemma:uniform_bound_for_cP}
Let ${\tilde m} \ge m \ge 0, \{r_k\}_{k=0}^{\tilde m}\subset\mathbb{R}_+$, $\calpha \in \cI$. Assume $\{\cP\}_{0<\calpha<1}$ is a uniformly bounded collection of RIFE projections. Then, there exists a constant $C$ independent of $\calpha$ such that
\begin{align}
\norm{\cP v}_{i, \cI} \leq C\norm{ v}_{m+1,\cI}, ~\forall v\in  \H^{m+1}_{\calpha,r}(\cI), ~0 \leq i \leq m+1.  \label{eqn:uniform_bound_for_cP_1}
\end{align}
\end{lemma}
\begin{proof}
Assume that $\{\cP\}_{0<\calpha<1}$ is a uniformly bounded collection of RIFE projections, then,
\begin{eqnarray*}
\norm{\cP v}_{0,\cI} \leq C\norm{ v}_{j,\cI}
\end{eqnarray*}
for an integer $j$ with $0 \leq j \leq m+1$. By \autoref{lemma:inverse_inequality}, we further have
\begin{align*}
\norm{\cP v}_{i,\cI} \leq c(m,r)\norm{\cP v}_{0,\cI} \leq c(m,r) C\norm{ v}_{j,\cI} \leq
c(m,r) C\norm{ v}_{m+1,\cI}, ~~0 \leq i \leq m+1
\end{align*}
which implies the uniform boundedness stated in \eqref{eqn:uniform_bound_for_cP_1}.

\end{proof}
%that the condition \eqref{eqn:def_uniformly_bounded} implies that $\norm{\cP %v}_{i,\cI}<\tilde{C}(m,r)\norm{ v}_{1,\cI}$ for $0\le i\le m$ since the inverse inequality in %\autoref{lemma:inverse_inequality} implies that $\norm{\cP v}_{1,\cI}<c(m,r)\norm{\cP v}_{0,\cI}$.
Now, we can derive an error bound for a collection of uniformly bounded LIFE projections that implies convergence.

\begin{theorem}
\label{thm:general_theorem_LIFE}
Let ${\tilde m} \ge m \ge 0, \{r_k\}_{k=0}^{\tilde m}\subset\mathbb{R}_+$, $\calpha \in \cI$. Assume that
$\{\Pr\}_{\alpha\in\overset{\circ}{I}_{k_0}}$ is an uniformly bounded collection of LIFE projections. Then, there exits a constant $C>0$ independent of $\alpha$ and $h$ such that
\begin{equation}
|v-\Pr v|_{i,\Ik}\le Ch^{m+1-i}|v|_{m+1,\Ik},\quad
\forall v\in \H_{\alpha,r}^{m+1}(\Ik),\quad \forall i=0,1,\dots,m.\label{eqn:general_theorem}
\end{equation}
\end{theorem}

\begin{proof}
By \autoref{lemma:uniform_bound_for_cP}, we know that
$\{\Pr\}_{\alpha\in\overset{\circ}{I}_{k_0}}$ satisfies \eqref{eqn:projection_error_reference_1} for
$i = 0, 1, 2, \dots, m$ with $j = m+1$. Consequently, we have
\begin{align*}
\norm{\cP}_{i, m+1, \cI}=\sup\left\{ \norm{\cP v }_{i,\cI}\mid v\in   \H^{m+1}_{\calpha,r}(\cI) \text{ and } \norm{v}_{m+1,\cI}=1\right\} \leq C.
\end{align*}
Then, applying these to \eqref{eqn:approximation_capabilities_general} established in
\autoref{thm:error_of_general_projection} yields \eqref{eqn:general_theorem}.

\end{proof}

The simplest example of a uniformly bounded collection of LIFE projections is the $L^2$ projection $\check{\text{Pr}}_{\calpha,r}^m:\H_{\calpha,r}^{m+1}(\cI)\to \cV$ defined by
$$\left(q,\check{\text{Pr}}_{\calpha,r}^m  v \right)_{\cI}=\left(q, v \right)_{\cI},\qquad \forall q\in \cV.$$
Choosing $q=\check{\text{Pr}}_{\calpha,r}^m  v$, we get $\norm{\check{\text{Pr}}_{\calpha,r}^m  v}_{0,\cI}=\norm{v}_{0,\cI}\le \norm{v}_{m+1,\cI}$. By \autoref{lemma:uniform_bound_for_cP}, $\{\check{\text{Pr}}_{\calpha,r}^m\}_{\alpha\in\overset{\circ}{I}_{k_0}}$
is an uniformly bounded collection of projections. Consequently, by  \autoref{thm:general_theorem_LIFE}, we can obtain the following optimal approximation capability of the LIFE space.

\begin{corollary}
\label{coro:L2_projection_rate}
Let ${\tilde m} \geq m\ge 0, \{r_k\}_{k=0}^{\tilde m}\subset\mathbb{R}_+$, $\calpha\in(0,1)$. Then, there exists a constant $C>0$ independent of $\alpha$ and $h$ such that
\begin{equation}
\min_{q\in \V^{m}_{\alpha,r}(\Ik)}|q- v|_{0,\Ik}
\le C h^{m+1}|v|_{m+1,\Ik},\quad \forall v\in \H_{\alpha,r}^{m+1}(\Ik). \label{eqn:L2_projection}
\end{equation}
\end{corollary}

% \section{Applications}
% \label{sec:applications}

% In this section, we employ the results of \autoref{sec:properties_of_v} and \autoref{sec:Bramble_Hilbert} to two  interface problems. First, we analyze a IFE-DG method for a large class of one dimensional hyperbolic interface problems. This class contains the acoustic interface problem \eqref{eqn:acoustic_acoustic_problem}, as well as other interface problems. The techniques that we use are parallel to the ones used to analyze DG methods for hyperbolic systems with constant coefficients  \cite{cockburnIntroductionDiscontinuousGalerkin1998}. More precisely, we define an immersed Radau projection, prove its existence and boundedness and use the scaling argument to obtain optimal error estimates. After that, we use a similar approach to construct a conforming IFE space for an elliptic interface problem and a stable Lobatto projection. The stable projection leads  to a direct application of Céa's lemma.

\section{Analysis of an IFE-DG method for a class of hyperbolic systems with discontinuous coefficients}\label{sec:one_D_class}

In this section, we employ the results of \autoref{sec:properties_of_v} and \autoref{sec:Bramble_Hilbert} to analyse an IFE-DG method for the acoustic interface problem \eqref{eqn:acoustic_acoustic_problem}. To our knowledge, the analysis of such method for hyperbolic systems has so far not been considered in the literature unlike the IFE methods for elliptic problems. The main challenge that time-dependent problems present is the plethora of possible jump coefficients. For instance, the jump coefficients $r^{p}_{k},r^{v}_k$ described in \eqref{eqn:extended_conditions_acoustic} for the acoustic interface problem are given by:
\begin{equation}
    r^p_{2k}=r^v_{2k}=\left(\frac{c_-}{c_+}\right)^{2k},\qquad r^{p}_{2k+1}= \frac{\rho_+}{\rho_-}\left(\frac{c_-}{c_+}\right)^{2k},\qquad r^{v}_{2k+1}= \frac{\rho_-}{\rho_+}\left(\frac{c_-}{c_+}\right)^{2k+2},\quad k=0,1,\dots
    \label{eqn:jump_coefficients_expanded}
\end{equation}
The nature of the jump coefficients in \eqref{eqn:jump_coefficients_expanded} makes the study of this particular IFE space extremely tedious as observed in \cite{moonImmersedDiscontinuousGalerkin2016}. Fortunately, the theory developed in \autoref{sec:properties_of_v} and \autoref{sec:Bramble_Hilbert} is general and applies to any choice of positive jump coefficients.

\subsection{Problem statement and preliminary results}

Let $I=(a,b)$ be a bounded interval containing $\alpha$, and let $\rho_{\pm},c_{\pm}$ be positive constants describing the density and the sound speed in $I^{\pm}$, respectively. Now, we consider the acoustic interface problem on $I$
\begin{subequations}\label{eqn:acoustic_matrix_form}
\begin{equation}\u_t(x,t)+A(x) \u_{x}(x,t)=0, \qquad x\in I\backslash \{\alpha\},\qquad t>0,
\label{eqn:acoustic_matrix_form:eqn}
\end{equation}
where $\u=(p,u)^T$ is the pressure-velocity couple and  \begin{equation} A_{\mid I^{\pm}}=A_{\pm}=\begin{pmatrix}
0& \rho_\pm c_{\pm}^2\\ \rho_{\pm}^{-1} & 0
\end{pmatrix} .\label{eqn:acoustic_matrix_form:A}\end{equation}
The matrices $A_\pm$ can be decomposed as $A_{\pm}=P_{\pm}\Lambda_{\pm} P_{\pm}^{-1}$, where $\Lambda_{\pm}=\diag(-c_\pm,c_\pm)$. Using this eigen-decomposition, we define $A_{\pm}^+=P_{\pm}\diag(0,c_{\pm}) P_{\pm}^{-1}$,  $A_{\pm}^-=P_{\pm}\diag(-c_{\pm},0) P_{\pm}^{-1}$, and $|A_{\pm}|=P_{\pm}\diag(c_{\pm},c_{\pm}) P_{\pm}^{-1}$ to be the positive part,the negative part, and the absolute value of $A_{\pm}$, respectively.
The acoustic interface problem that we are considering here is subject to the following homogeneous inflow boundary conditions
\begin{equation}
A_-^+\u(a,t)=A_+^-\u(b,t)=0, \qquad t\ge 0,
\label{eqn:acoustic_matrix_form:BC}
\end{equation}
initial conditions
\begin{equation}
\u(x,0)=\u_0(x), x\in I,
\label{eqn:acoustic_matrix_form:initial}
\end{equation}
and interface condition
\begin{equation}
\u(\alpha^-,t)=\u(\alpha^+,t),\qquad   t\ge 0.
\label{eqn:acoustic_matrix_form:IC}
\end{equation}
\end{subequations}
In the remainder of this section, let $S_{\pm}=\diag(\rho_{\pm}^{-1}c_{\pm}^{-2},\rho_{\pm})$ and $S(x)=S_{\pm}$ if $x\in I^{\pm}$, then
\begin{equation}
    S_{\pm}A_{\pm} =  \begin{pmatrix}
        0&1\\
        1&0
    \end{pmatrix}=\tilde{A}.
    \label{eqn:definition_of_tilde_A}
\end{equation}
Now, we can multiply \eqref{eqn:acoustic_matrix_form:eqn} by $S$ and write the acoustic interface problem as
\begin{equation}
    S(x)\u_t(x,t)+\tilde{A}\u_{x}(x,t)=0, \qquad x\in I\backslash \{\alpha\},\qquad t>0.
\label{eqn:acoustic_symmetric_frederichs}
\end{equation}

Lombard and Piraux \cite{lombardModelisationNumeriquePropagation2002}, and Moon \cite{moonImmersedDiscontinuousGalerkin2016} have shown that by successively differentiating $\bb{\u(\cdot,t)}_{\alpha}=0$, where $\bb{\cdot}$ is the jump, we obtain
$$\bb{A(\cdot)^k\u(\cdot,t)}_{\alpha}= 0 \Longleftrightarrow \frac{\partial^k}{\partial x^k}\u(\alpha^+,t)=R_k \frac{\partial^k}{\partial x^k}\u(\alpha^-,t),\quad R_k=A_+^{-k}A_-^k, \ k=0,1,\dots,m.$$
Since $R_k$ is diagonal (see part (a) of \autoref{lem:properties_of_matrices}), the condition $\u^{(k)}(\alpha^+,t)=R_k \u^{(k)}(\alpha^-,t)$ is equivalent to

\begin{equation}\frac{\partial ^k}{\partial x^k}p(\alpha^+,t)=r^{p}_k\frac{ \partial^k}{\partial x^k} p(\alpha^-,t),\qquad \frac{\partial ^k}{\partial x^k}u(\alpha^+,t)=r^{u}_k\frac{ \partial^k}{\partial x^k} u(\alpha^-,t),\label{eqn:def_of_R}\end{equation}
where $r^p_k$ and $r^u_k$ are defined in \eqref{eqn:jump_coefficients_expanded}. These decoupled interface conditions make the results obtained previously about the approximation capabilities of the LIFE space directly applicable to vector functions in the product spaces
$$ \HH^{m+1}_{\alpha,\r}(\Ik)
=\H^{m+1}_{\alpha,r^{p}} (\Ik)\times  \H^{m+1}_{\alpha,r^{u}} (\Ik)
,\quad
\VV^{m}_{\alpha,\r}(\Ik)=
 \V^m_{\alpha,r^{p}} (\Ik)\times \V^m_{\alpha,r^{u}} (\Ik),
\quad \WW^{m}_{\alpha,\r}(\T_h)=
W^{m}_{\alpha,r^{p}}(\T_h)\times W^{m}_{\alpha,r^{u}}(\T_h) ,
$$
where $\r=(r^p,r^u)$. Now, we define the following bilinear forms

\begin{subequations}
     \label{eqn:definition_of_M_and_B}
\begin{equation}
    M(\w,\v)=\sum_{k=1}^N \left(S\v,\w\right)_{I_k}
     \label{eqn:definition_of_M}
    \end{equation}
    \begin{equation}
    B(\w,\v)= \sum_{k=1}^N\left(\v',\tA\w\right)_{I_k}+\sum_{k=0}^N \bb{\v}_{x_k}^T S(x_{k})\hat{\w}(x_{k}),
    \label{eqn:definition_of_B}
\end{equation}

\end{subequations}

where the numerical flux  $\hat{\w}(x_k)=A(x_k)^+\w(x_k^-)+A(x_k)^-\w(x_k^+)$ at the interior nodes. At the boundary, we have $\hat{\w}(x_N)= A(x_N)^+\w(x_N^-)$, $\hat{\w}(x_0)= A(x_0)^-\w(x_0^+)$, $\bb{\w}_{x_N}=\w(x_N^-)$ and  $\bb{\w}_{x_0}=-\w(x_0^+)$.

Now we define the immersed DG formulation as: Find $\u_h\in C^1\left([0,T], \WW^{m}_{\alpha,\r}(\T_h)\right)$ such that
\begin{subequations}\label{eqn:general_IDG_form_compact}
\begin{align}
(\u_h(\cdot,0),\v_h)_{I}=(\u_0,\v_h)_{I},\qquad &\forall\v_h \in \WW^{m}_{\alpha,\r}(\T_h),\label{eqn:general_IDG_form_compact:initial_projection}\\
M(\u_{h,t}(\cdot,t),\v_h)=
B(\u_h(\cdot,t),\v_h),\qquad &\forall \v_h \in \WW^{m}_{\alpha,\r}(\T_h) ,\label{eqn:general_IDG_form_compact:Weak_form}
\end{align}
\end{subequations}

We note that the discrete weak form \eqref{eqn:general_IDG_form_compact} and the discrete space $\WW^{m}_{\alpha,\r}(\T_h)$ are identical to the ones described in the IDPGFE formulation in \cite{moonImmersedDiscontinuousGalerkin2016}.

Next, we will go through some basic properties of the matrices $S_{\pm}$ and $A_{\pm}$, these properties will be used later in the proof the $L^2$ stability in \autoref{eqn:continuous_stability}, in the analysis of the immersed Radau projection and in the convergence estimate.

\begin{lemma}
\label{lem:properties_of_matrices}
Let $A_{\pm}$ be the matrices defined in \eqref{eqn:acoustic_matrix_form:A} and let $S_{\pm}= \diag(\rho_{\pm}^{-1}c_{\pm}^{-1},\rho_{\pm})$, then

\begin{enumerate}[label=(\alph*)]
    \item For any integer $k\ge 0$, the matrix $A_{+}^{-k}A_{-}^k$ is diagonal with positive entries.
    \item Let $s\in \{+,-\}$, then there is an invertible matrix $P_s$ such that $A_s=P_s\diag(-c_s,c_s) P_{s}^{-1}$ and $S_s=P_{s}^{-T}P_s^{-1}$.
    \item Let $s\in \{+,-\}$, then the matrices $S_sA_s^+,\ S_sA_s^-$ and $S_s|A_s|$ are symmetric. Furthermore, $S_sA^+_s$ is positive semi-definite, $S_sA^-_s$ is negative semi-definite and $S_s|A_s|$ is positive definite.
    \item Let $s,\tilde{s}\in \{+,-\}$, and let $\w\in \mathbb{R}^2$, then

  \begin{equation} \left(\norm{A^{\tilde{s}}_s\w}^2+\left|\w^T S_sA^{\tilde{s}'}_s\w\right|=0\right)\Longrightarrow \w=0.
    \label{eqn:SA_inequality}
  \end{equation}
  where $\tilde{s}'$ is dual of $\tilde{s}$ defined at the beginning of \autoref{sec:properties_of_v}, and $\norm{\cdot}$ is Euclidean norm.
  \item Let $s\in\{+,-\}$. Then, there is a constant $C(\rho_{s},c_{s})>0$ such that
  \begin{equation}
    \w^TS_sA_s^+ \w-\w^T S_s A_s^-\w \ge C(\rho_s,c_s)\norm{\w}^2,\qquad \forall \w\in \mathbb{R}^2.
      \label{eqn:SA_vector_norm}
  \end{equation}
    \end{enumerate}

\end{lemma}

\begin{proof}
    \begin{enumerate}[label=(\alph*)]
     \item We have $A_{\pm}^2=c_{\pm}^2 \Id_2$, where $\Id_2$ is the the $2\times 2$ identity matrix. Therefore,
     \begin{equation}
         A_{\pm}^{2k}=c_{\pm}^{2k} \Id_2,\qquad  A_{\pm}^{2k+1}=c_{\pm}^{2k} A_{\pm},\qquad k=0,1,\dots \label{eqn:powers_of_A}
     \end{equation}
     Using \eqref{eqn:powers_of_A}, we immediately obtain  $A_+^{-2k}A_{-}^{2k}=\left(\frac{c_-}{c_+}\right)^{2k} \Id_2$ and $A_+^{-2k-1}A_{-}^{2k+1}=\left(\frac{c_-}{c_+}\right)^{2k}A_+^{-1}A_-$. Finally, by direct computation, we have $A_+^{-1}A_-=\diag\left(\frac{\rho_+}{\rho_-},\frac{\rho_-c_-^2}{\rho_+c_+^2}\right)$. Hence, $A_{+}^{-k}A_{-}^k= \diag(r^p_k,r^u_k)$, where $r^p_k$ and $r^u_k$are defined in \eqref{eqn:jump_coefficients_expanded}.

     \item Let \begin{equation}P_s=\frac{1}{\sqrt{2\rho_s}} \begin{pmatrix}
     -c_s\rho_s & c_s \rho_s\\ 1& 1
     \end{pmatrix},\label{eqn:definition_of_P}\end{equation}
     then, by a simple computation, we can show that $S_s=P_s^{-T}P_s^{-1}$ and $A_s=P_s \diag(-c_s,c_s)P_s^{-1}$.
     \item We have $S_sA_s^{+}=P_s^{-T}\diag(0,c_s)P_s^{-1}$, where $P_s$ is defined in \eqref{eqn:definition_of_P}. Therefore, $S_sA_s^{+}$ is a symmetric semi-positive definite matrix. The other two claims can be proven similarly.
     \item We will only consider the case $\tilde{s}=+$ here, the other case can be proven similarly. Consider a vector $\w\in \mathbb{R}^2$ that satisfies
     \begin{equation}\norm{A^{+}_s\w}^2+\left|\w^T S_sA^{-}_s\w\right|=0.
     \label{eqn:pseudo_norm_equal_to_zero}
     \end{equation}
     Now, let $\tilde{\w}=P_s\w$ where $P_s$ is defined in \eqref{eqn:definition_of_P}, then \eqref{eqn:pseudo_norm_equal_to_zero} can be written as
     $$ \norm{P_s\diag(0,c_s)\tilde{\w}}^2 + \norm{\diag(-c_s,0)\tilde{\w}}^2=0, $$
     Since both norms are non-negative, we have
   $\diag(-c_s,0)\tilde{\w}=0$  and $P_s\diag(0,c_s)\tilde{\w}=0$.  $P_s$ being invertible, we get $\tilde{\w}=0$. Consequently, $\w=P_s^{-1}\tilde{\w}=0$.

   \item We have by direct computation $$\w^TS_s(A_s^+-A_s^-)\w  = \w^T\begin{pmatrix}
       \rho_s^{-1}c_s^{-1}& 0\\
       0 &\rho_s c_s
   \end{pmatrix}\w\ge \min(\rho_sc_s,\rho_s^{-1}c_s^{-1})\norm{\w}. $$
    \end{enumerate}

\end{proof}

\begin{lemma}
\label{eqn:continuous_stability}
Let $\u$ be a solution to \eqref{eqn:acoustic_matrix_form}, and let $\epsilon(t)= \frac{1}{2}\left(\u(\cdot,t),S(\cdot)\u(\cdot,t)\right)_{0,I}^2 $, then
$$ \epsilon'(t) \le 0,\qquad t\ge  0.$$
\end{lemma}

\begin{proof}
    By multiplying \eqref{eqn:acoustic_symmetric_frederichs} by $\u^T$ and integrating on $I^\pm$, we obtain
    $$\int_I \u(x,t)^TS\u_t(x,t) +\u(x,t)^T\tilde{A}\u_x(x,t)\ dx=0.$$
    The matrices $S$ and $\tilde{A}$ are symmetric. Therefore, we rewrite the previous equation as
    $$\epsilon'(t)+\sum_{s=+,-}\int_{I^s}\frac{\partial}{\partial x}\u(x,t)^T\tilde{A}\u(x,t)) dx=0. $$
Since $\u$ is continuous at $\alpha$ (from \eqref{eqn:acoustic_matrix_form:IC}), we have
\begin{equation}\epsilon'(t)+\u(b,t)^T\tilde{A}\u(b,t)-
\u(a,t)^T\tilde{A}\u(a,t)=0.\label{eqn:stability_with_boundary}\end{equation}
Now, we can rewrite the term $\u(b,t)^T\tilde{A}\u(b,t)$ as
$$\u(b,t)^T\tilde{A}\u(b,t)=\u(b,t)^TS_+A_+\u(b,t)=\u(b,t)^TS_+A_+^+\u(b,t)+\u(b,t)^TS_+A^-_+\u(b,t)=\u(b,t)^TS_+A_+^+\u(b,t),$$
where the last equality follows from the boundary condition \eqref{eqn:acoustic_matrix_form:BC}. Since $SA^+$ is symmetric semi-positive definite (see part (c) of \autoref{lem:properties_of_matrices}), we conclude that $\u(b,t)^T\tilde{A}\u(b,t)\ge 0$. Similarly, we have  $\u(a,t)^T\tilde{A}\u(a,t)\le 0$. Therefore,
$$\epsilon'(t)\le 0.$$
\end{proof}
% \end{proof}

The previous lemma shows that $\epsilon(t)$, interpreted as the energy of the system, is decreasing. This is to be expected since the boundary conditions in \eqref{eqn:acoustic_matrix_form:BC} are dissipative (see \cite{benzoni-gavageMultidimensionalHyperbolicPartial2006}). Furthermore, if we let $\epsilon_h(t)=\frac{1}{2}M(\u_h(\cdot,t),\u_h(\cdot,t))$ be the discrete energy, then
\begin{equation}\epsilon_h'(t)=B(\u_h,\u_h)= \frac{-1}{2} \sum_{k=0}^N \bb{\u_h}^T_{x_k}S(x_k)|A(x_k)|\bb{\u_h}_{x_k}\le 0.
\label{eqn:negative_definite_form}
\end{equation}
The proof of \eqref{eqn:negative_definite_form} follows the same steps as the scalar case described in \cite{cockburnIntroductionDiscontinuousGalerkin1998}.

\subsection{The immersed Radau projection and the convergence analysis}

We denote by $\mathcal{R}\u\in \WW^m_{\alpha,\r}(\T_h)$ the global Gauss-Radau projection defined as
\begin{equation}B(\u - \mathcal{R}\u,\v_h) =0,\qquad  \forall \v_h\in  \WW^m_{\alpha,\r}(\T_h).\label{eqn:Global_Radau_projection}\end{equation}

Although $\mathcal{R}\u$ is a global projection, it can be constructed on each element independently, the construction of $(\mathcal{R}\u)_{\mid I_k}$ where $k\ne k_0$ can be found in \cite{adjeridAsymptoticallyExactPosteriori2010,yangAnalysisOptimalSuperconvergence2012} for the scalar case and can be generalized easily to systems. On the interface element, we define the local immersed Radau projection (\textit{IRP}) operator $\Pi^m_{\alpha,\r}: \HH^{m+1}_{\alpha,\r}(\Ik)\to  \VV^{m}_{\alpha,\r}(\Ik)$ using \eqref{eqn:phys_ref_diagram} as
$$\Pi^m_{\alpha,\r}=\M^{-1}\circ \cPi \circ \M,$$
% \begin{equation}
%     \begin{cases}
%     A^{-}_- \Pi^m_{\alpha,\r} \u(x_{k_0-1}) =  A^{-}_- \u(x_{k_0-1}),\\
%     A^{+}_+ \Pi^m_{\alpha,\r} \u(x_{k_0}) =  A^{+}_+ \u(x_{k_0}),\\
%     \left(\tilde A\v_h', \Pi^m_{\alpha,\r}\u\right)_{\Ik}= \left(\tilde A\v_h', \u\right)_{\Ik},&\forall \v_h\in \VV^{m}_{\alpha,\r}(\Ik).
%     \end{cases}
%     \label{eqn:first_formulation_of_Radau}
% \end{equation}
 where $\cPi : \HH^{m+1}_{\calpha,\r}(\cI)\to  \cVV$ is called the reference IRP operator and it is defined as the solution to the following system:
 \begin{subequations}
     \label{eqn:reference_formulation_of_Radau}
\begin{align}[left ={\empheqlbrace}]
    A^{-}_- \cPi\cuu(0) &=  A^{-}_- \cuu(0),    \label{eqn:reference_formulation_of_Radau:a}\\
    A^{+}_+ \cPi \cuu(1) &=  A^{+}_+ \cuu(1),    \label{eqn:reference_formulation_of_Radau:b}\\
    \left( \tilde A\v',\cPi\cuu\right)_{\cI}&= \left(\tilde A\v', \cuu\right)_{\cI},\qquad \forall \v\in \cVV.    \label{eqn:reference_formulation_of_Radau:c}
\end{align}
 \end{subequations}

 Next,  we will go through some basic properties of the IRP to prove that the IRP is well defined and is uniformly bounded on the  RIFE space $\cVV$. From there, we can show that the IRP error on the  LIFE space $\VV^{m}_{\alpha,\r}(\Ik)$  decays at an optimal rate of $O(h^{m+1})$ under mesh refinement.

\begin{lemma} \label{lem:commutative_diagram}
 Let $A$ be the matrix function defined in \eqref{eqn:acoustic_matrix_form:A} and let $\pp\in \cVV$, then $A\pp' \in \mathbb{V}^{m-1}_{\calpha,\r}(\cI)$. Furthermore the map
 \begin{align*}
     G:\cVV&\to \mathbb{V}^{m-1}_{\calpha,\r}(\cI)\\
     \pp &\mapsto  A\pp'
 \end{align*}
 is surjective.
\end{lemma}

\begin{proof}Let $\pp\in \cVV$ and let $\tilde{\pp}=A\pp'$, then for a fixed $k\in\{0,1,\dots,m-1\}$, we have
\begin{align}\tilde{\pp}^{(k)}(\calpha^+)&= A_+\pp^{(k+1)}(\calpha^+)\tag{Using  $\tilde{\pp}=A\pp'$ }  \\
&=A_+A_+^{-k-1}A_{-}^{k+1}\pp^{(k+1)}(\calpha^-)\tag{By construction of $ \cVV$}\\
&= A_+A_+^{-k-1}A_{-}^{k+1}A_-^{-1}\tilde{\pp}^{(k)}(\calpha^-), \tag{Using $\pp'=A^{-1}\tilde{\pp}$}\\
&=A_+^{-k}A_-^k \tilde{\pp}^{(k)}(\calpha^-).\label{eqn:p_tilde_jump}
\end{align}
Since \eqref{eqn:p_tilde_jump} holds for every $k=0,1,\dots,m-1$, we conclude that $\tilde{\pp}\in \mathbb{V}^{m-1}_{\calpha,\r}(\cI)$.

Now, we show that $G$ is surjective, by the rank-nullity theorem, it suffices to prove that  $\dim\mathrm{ker}(G)$ is $2$ since $\dim \cVV - \dim \mathbb{V}^{m-1}_{\calpha,\r}(\cI)=2$. Let $\pp \in \mathrm{ker}(G)$, then $A\pp'=0$, since $A$ is invertible, we get $\pp'=0$ which implies that $\pp \in\mathbb{V}^{0}_{\calpha,\r}(\cI)$. This shows that $\dim\mathrm{ker}(G)=\dim \mathbb{V}^{0}_{\calpha,\r}(\cI)=2$.
\end{proof}

Following the definition of $G$, we can re-write \eqref{eqn:reference_formulation_of_Radau:c} as
$$(SG(\v),\cPi \cuu)_{\cI}= (SG(\v),\cuu)_{\cI},\qquad \forall \v\in \cVV .$$
For convenience, we will write $(SG(\v),\cPi \cuu)_{\cI}=(G(\v),\cPi \cuu)_{S,\cI}$. Now, since $G$ maps $\cVV$ onto $\mathbb{V}^{m-1}_{\calpha,\r}(\cI)$, we can express the condition  \eqref{eqn:reference_formulation_of_Radau:c} as
\begin{equation}(\v,\cPi\cuu)_{S,\cI} =(\v,\cuu)_{S,\cI},\qquad \forall\v\in \mathbb{V}^{m-1}_{\calpha,\r}(\cI).\label{eqn:second_form_orthogonality}\end{equation}

% \noindent \textbf{Comments: } The previous lemma shows that the last condition in \eqref{eqn:reference_formulation_of_Radau} is equivalent to
% \begin{equation}\left(  \v,\cPi\u\right)_{K,\cI}= \left(\v, \u\right)_{\cI},\qquad \forall \v\in \VV^{m-1}_{\calpha,\r}(\cI),
% \label{eqn:second_form_orthogonality}
% \end{equation}
% where $(\cdot,\cdot)_{K,\cI}$ is the weighted inner product with the weight  function $K$ defined in \eqref{eqn:def_K}. Also, by \autoref{lem:commutative_diagram}, the following  diagram commutes:
% \begin{equation}
%     \includegraphics[valign=c]{figures/Commutative diagram.pdf}
%     \label{eqn:commutative_diagram}
% \end{equation}
% Here, $\tau(\r)=(\tau(r^{(1)}),\tau(r^{(2)},\dots,\tau(r^{(n)}))$ and $\tilde{\r}= (\tilde{r}^{(1)},\tilde{r}^{(2)},\dots,\tilde{r}^{(n)})$  where $\tilde{r}^{(i)}_k=(S_+S_-^{-1})_{ii}r^{(i)}_{k}$ (see \eqref{eqn:q_scaling}).

\begin{theorem}\label{thm:Radau_uniqueness}
The system \eqref{eqn:reference_formulation_of_Radau} admits  exactly one solution.
\end{theorem}

\begin{proof} First, we prove that the system admits at most one solution, for that,  we only need to show that if $\cuu=0$, then $\cPi\cuu=0$. For simplicity, let $\q=(q_1,q_2)^T$ be the solution to \eqref{eqn:reference_formulation_of_Radau} with $\cuu=0$, and let $r^{(1)}=r^p$ and $r^{(2)}=r^u$,  then by \eqref{eqn:second_form_orthogonality}, we have
\begin{equation}(v,q_i)_{w_i,\cI} =0,\ \forall v\in \V^{m-1}_{\calpha,r^{(i)}}(\cI),\text{ where } w_i=S_{i,i}, \qquad i=1,2,\label{eqn:q_orthogonal}\end{equation}
which is equivalent to $q_i\in \mathcal{Q}^{m}_{\calpha,w_i,r^{(i)}}(\cI)$. On the other hand, we have
\begin{equation}\left( \tilde A\q',\q \right)_{\cI} = \frac{1}{2}\left(\q(1)^TS_+ A_+\q(1)-\q(0)^T S_-A_-\q(0)\right)=0.\label{eqn:q_integral}
\end{equation}

From \eqref{eqn:reference_formulation_of_Radau:a} and \eqref{eqn:reference_formulation_of_Radau:b}, we have $A^+_+\q(1)=A^-_-\q(0)=0$, then $ A_+\q(1)=A_+^-\q(1)$ and $A_-\q(0)=A_-^+\q(0)$. Therefore, the equation \eqref{eqn:q_integral} becomes
$$\frac{1}{2}\left(\q(1)^TS_+ A_+^-\q(1)-\q(0)^T S_-A_-^+\q(0)\right)=0.$$
Now, by \autoref{lem:properties_of_matrices} part (c), the quantities $\q(1)^TS_+ A_+^-\q(1)$ and $-\q(0)^T S_-A_-^+\q(0)$ are non-positive, then
$$\q(1)^TS_+ A_+^-\q(1)=\q(0)^T S_-A_-^+\q(0)=0,$$
Furthermore, by \eqref{eqn:reference_formulation_of_Radau:a} and \eqref{eqn:reference_formulation_of_Radau:b}, we have
$$\norm{A_+^+\q(1)}^2+|\q(1)^TS_+ A_+^-\q(1)|=\norm{A_-^-\q(0)}^2+|\q(0)^T S_-A_-^+\q(0)|=0.$$
At this point, we use \autoref{lem:properties_of_matrices} part (d) to conclude that $\q(1)=\q(0)=0$. Therefore, $q_i$ are orthogonal IFE functions (as shown in \eqref{eqn:q_orthogonal}) that vanish on the boundary. By \autoref{thm:Zeros of orthogonal IFE functions}, we conclude that $q_i\equiv 0$ for  $i=1,2$. Equivalently, $\q\equiv 0$.

To finalize the proof, we only need to show that \eqref{eqn:reference_formulation_of_Radau} can be written as a square system. Let $A_{\pm}=P_\pm\diag(-c_\pm,c_\pm)P_{\pm}^{-1}$ be an eigen-decomposition of $A_{\pm}$. Then, \eqref{eqn:reference_formulation_of_Radau} can be written as
\begin{equation}
    \begin{cases}
    \left(P_-^{-1} \cPi\cuu(0)\right)_1 &=  \left(P_-^{-1}\cuu(0)\right)_1\\
    \left(P_+^{-1} \cPi\cuu(1)\right)_2 &=  \left(P_+ ^{-1}\cuu(1)\right)_2,\\
    \left( \N^j_{\calpha,r^{p}},(\cPi\cuu)_i\right)_{S_{11},\cI}&=  \left( \N^j_{\calpha,r^{p}},\check{p}\right)_{S_{22},\cI}, \qquad \ 1\le j\le m-1,\\
     \left( \N^j_{\calpha,r^{u}},(\cPi\cuu)_i\right)_{S_{22},\cI}&=  \left( \N^j_{\calpha,r^{u}},\check{u}\right)_{S_{22},\cI}, \qquad \ 1\le j\le m-1
    \end{cases}
    \label{eqn:reference_formulation_of_Radau_with_R}
\end{equation}
which is a system of $2(m+1)$ equations with $2(m+1)$ variables. Since the homogeneous system admits at most one solution, we conclude that \eqref{eqn:reference_formulation_of_Radau} has exactly one solution.

\end{proof}

Next, we show that $\{\cPi\}_{0<\calpha<1}$ is uniformly bounded. First, let  $\pp\in\VV^{m-1}_{\calpha,\r}(\cI) $ be the solution to the following symmetric positive definite system
\begin{equation}
     \left(  \v,\pp\right)_{S,\cI}= \left(\v, \cuu\right)_{\cI},\ \forall \v\in \VV^{m-1}_{\calpha,\r}(\cI),\label{eqn:p_definition}
\end{equation}
and let $\q =\cPi \cuu -\pp$, then by \eqref{eqn:reference_formulation_of_Radau:c} and \eqref{eqn:second_form_orthogonality}, we have

\begin{equation}\left(  \tilde{A}\v',\q\right)_{\cI}=\left(  \tilde{A}\v',\cPi\cuu-\pp\right)_{\cI}=0,\qquad \forall \v\in \VV^{m}_{\calpha,\r}(\cI),
\label{eqn:q_is_in_Q}
\end{equation}
which can be written as
$$\left(  \v,\q\right)_{S,\cI}=0,\qquad \forall\v\in \VV^{m-1}_{\calpha,\r}(\cI). $$
Thus, $\q\in \cQQ=\mathcal{Q}^m_{\calpha,S_{11},r^p}(\cI)\times \mathcal{Q}^m_{\calpha,S_{22},r^u}(\cI)$. Additionally, by \eqref{eqn:reference_formulation_of_Radau:a} and \eqref{eqn:reference_formulation_of_Radau:b}, we have

    \begin{equation}    \begin{cases}
    A^{-}_- \q(0) =  A^{-}_- \left(\cuu(0)-\pp(0)\right),\\
    A^{+}_+ \q(1) =  A^{+}_+ \left(\cuu(1)-\pp(1)\right).\end{cases}
    \label{eqn:trace_of_q}
    \end{equation}

    % We first show that the norm of $\pp$ is bounded independently of $\calpha$.

    In the next two lemmas, we prove that $\norm{\pp}_{0,\cI}$ and $\norm{\q}_{0,\cI}$ is bounded by some appropriate norms of ${\color{red}\cuu}$ independently of $\calpha$. Both lemmas will be used later in \autoref{thm:Radau_stability} to prove that $\{\cPi\}_{0<\calpha<1}$ is a uniformly bounded collection of RIFE projections.

    \begin{lemma}\label{lem:p_norm}
 Let $\cuu\in\HH^{m+1}_{\calpha,r}(\cI)$ and $\pp\in\VV^{m-1}_{\calpha,\r}(\cI)$ defined by \eqref{eqn:p_definition}, then there is $C(\rho,c)>0$ independent of $\calpha$ such that  $\norm{\pp}_{0,\cI}\le C(\rho,c) \norm{\cuu}_{0,\cI}$.
    \end{lemma}
    \begin{proof}
    Let $\v=\pp$ in \eqref{eqn:p_definition}, then $\norm{\pp^TS\pp}_{0,\cI}=(\pp,\cuu)_{0,\cI}\le \norm{\pp}_{0,\cI}\norm{\cuu}_{0,\cI}$. On the other hand, by construction of $S$, we have $\norm{\pp^TS\pp}_{0,\cI}\ge C(\rho,c)\norm{\pp}_{0,\cI}$, where $C(\rho,c)=\min(\rho_{-}^{-1}c_-^{-2}, \rho_+^{-1}c_+^{-2}, \rho_{-},\rho_+)$. Therefore,
    $$\norm{\pp}_{0,\cI}\le C(\rho,c)\norm{\cuu}_{0,\cI}.$$
    \end{proof}

 \begin{lemma}\label{lem:q_norm}
     Let $\cuu\in\HH^{m+1}_{\calpha,r}(\cI)$ and let $\pp\in\VV^{m-1}_{\calpha,\r}(\cI)$ defined by \eqref{eqn:p_definition}. Then, there is $C(\rho,c,m)>0 $ independent of $\calpha$ such that $$\norm{\cPi\cuu-\pp}_{0,\cI}\le C(\rho,c,m)\norm{\cuu}_{1,\cI}$$
    \end{lemma}
\begin{proof}
    Let $\q=(q_1,q_2)=\cPi\cuu-\pp$, then by \eqref{eqn:q_is_in_Q}, we have $(\tilde{A}\q',\q)_{\cI}=0$ which is equivalent to
    \begin{equation}\q(1)^T\tilde{A}\q(1) -\q(0)^T \tilde{A}\q(0)=0
    \label{eqn:q1_q0}\end{equation}
    By decomposing $A_{\pm}=A^+_{\pm}+A^-_\pm$ and using the definition of $\tilde{A}$ in \eqref{eqn:definition_of_tilde_A}, we can split \eqref{eqn:q1_q0} as
    \begin{equation}\q(1)^TS_+A_+^+\q(1) +\q(1)^TS_+A_+^-\q(1)-\q(0)^T S_-A_-^+\q(0)-\q(0)^T S_-A_-^-\q(0)=0.
    \label{eqn:q1_q0_expanded}\end{equation}
    Now, from \eqref{eqn:SA_vector_norm}, we have

    \begin{subequations}\label{eqn:SA_q_norm}
        \begin{equation}
             \q(0)^TS_-A_-^+ \q(0)-\q(0)^T S_- A_-^-\q(0) \ge C_1(\rho,c)\norm{\q(0)}^2,
             \label{eqn:SA_q_norm:0}
        \end{equation}
        \begin{equation}
             \q(1)^TS_+A_+^+ \q(1)-\q(1)^T S_+ A_+^-\q(1) \ge C_1(\rho,c)\norm{\q(1)}^2.
             \label{eqn:SA_q_norm:1}
        \end{equation}
    \end{subequations}
    Next, we sum \eqref{eqn:q1_q0_expanded}, \eqref{eqn:SA_q_norm:0} and \eqref{eqn:SA_q_norm:1} to obtain

      \begin{equation}
            \q(1)S_+A_+^+ \q(1)-\q(0)S_-A_-^- \q(0)\ge \frac{1}{2}C_1(\rho,c)\left(\norm{\q(0)}^2+\norm{\q(1)}^2\right)
        \label{eqn:q_before_trace}\end{equation}
    We substitute \eqref{eqn:trace_of_q} in \eqref{eqn:q_before_trace} to obtain
    \begin{equation}
            \q(1)S_+A_+^+ (\cuu(1)-\pp(1))-\q(0)S_-A_-^- (\cuu(0)-\pp(0))\ge \frac{1}{2}C_1(\rho,c)\left(\norm{\q(0)}^2+\norm{\q(1)}^2\right)
        \label{eqn:q_after_trace}\end{equation}
Now, we will bound the left hand side from above. First, we have
\begin{equation}\q(1)S_+A_+^+ (\cuu(1)-\pp(1))-\q(0)S_-A_-^- (\cuu(0)-\pp(0))\le C_2(\rho,c) \left(\norm{\q(1)}\norm{\cuu(1)-\pp(1)}+\norm{\q(0)}\norm{\cuu(0)-\pp(0)}\right)\label{eqn:remove_SA}\end{equation}
Since $\cuu-\pp\in (H^1(\cI))^2$, there is $C_3>0$ such that
\begin{equation}
    \max(\norm{\cuu(0)-\pp(0)}, \norm{\cuu(1)-\pp(1)})\le C_3\left(\norm{\cuu}_{1,\cI}+\norm{\pp}_{1,\cI}\right).
    \label{eqn:u_check_minus_p}
\end{equation}
By applying the inverse inequality  \eqref{eqn:inverse_inequality} and \autoref{lem:p_norm} to $\norm{\pp}_{1,\cI}$, we obtain
\begin{align}
     \max(\norm{\cuu(0)-\pp(0)}, \norm{\cuu(1)-\pp(1)})
     &\le C_4(\rho,c,m)\left(\norm{\cuu}_{1,\cI}+\norm{\pp}_{0,\cI}\right),\notag
     \\ &\le C_5(\rho,c,m)\norm{\cuu}_{1,\cI}
    \label{eqn:u_check_minus_p_2}
\end{align}
Now, we substitute \eqref{eqn:u_check_minus_p_2} and \eqref{eqn:remove_SA} back into \eqref{eqn:q_after_trace} and use the inequality $a^2+b^2 \ge \frac{1}{2}(a+b)^2$ to obtain

$$\left(\norm{\q(0)}+\norm{\q(1)}\right)\norm{\cuu}_{1,\cI}\ge
C_6(\rho,c,m)\left(\norm{\q(0)}+\norm{\q(1)}\right)^2,$$
which yields
\begin{equation}\norm{\q(0)}+\norm{\q(1)}\le
C_7(\rho,c,m)\norm{\cuu}_{1,\cI}.
\label{eqn:q0_q1_before_trace}
\end{equation}

To finish the proof, we recall that $\q\in \mathcal{Q}^m_{\calpha,S_{11},r^p}(\cI)\times \mathcal{Q}^m_{\calpha,S_{22},r^u}(\cI)$. Therefore, by \autoref{lem:norm_on_Q} and some elementary algebraic manipulations, we have

\begin{align}
\norm{\q}_{0,\cI}&=\sqrt{\norm{q_1}_{0,\cI}^2 +\norm{q_2}_{0,\cI}^2}\tag{By defintion} \\
&\le C_7(\rho,c,m)\sqrt{q_1(0)^2+q_1(1)^2+q_2(0)^2+q_2(1)^2},\tag{Using \autoref{lem:norm_on_Q}}\\
&\le C_8(\rho,c,m)\left(\norm{\q(0)}+\norm{\q(1)}\right)\notag\\
\norm{\q}_{0,\cI}&\le C_9(\rho,c,m)\norm{\cuu}_{1,\cI},\tag{using \eqref{eqn:q0_q1_before_trace}}
\end{align}
which is the desired result.
\end{proof}

By combining \autoref{lem:p_norm} and \autoref{lem:q_norm}, we can show that the norm of $\cPi\cuu$ can be bounded by a norm of $\cuu$ independently of $\calpha$ as described in the following theorem. We note that $\cPi$ maps $\HH^{m+1}_{\calpha,\r}(\cI)$ to $\cVV$. Nevertheless, we shall call $\cPi$ a RIFE projection since the results from \autoref{sec:Bramble_Hilbert} in the scalar case apply directly to the vector case here.

    \begin{theorem}\label{thm:Radau_stability}
    Let $m\ge 1$ and let $\cuu\in \HH^{m+1}_{\calpha,\r}(\cI)$. Then, there is $C(\rho,c,m)>0$ independent of $\calpha$ such that
    $$\norm{\cPi\cuu}_{0,\cI} \le C(\rho,c,m)\norm{\cuu}_{1,\cI}$$
    That is, $\{\cPi\}_{0<\calpha<1} $ is a uniformly bounded collection of RIFE projections.
    \end{theorem}
      \begin{proof}
    % This is a direct consequence of  \autoref{lem:p_norm} and \autoref{lem:q_norm}.
    By definition of $\pp$ and $\q$, we have
    $$ \norm{\cPi\cuu}_{0,\cI}\le \norm{\pp}_{0,\cI}+\norm{\q}_{0,\cI},$$
    Which, using \autoref{lem:p_norm} and \autoref{lem:q_norm}, leads to
$$ \norm{\cPi\cuu}_{0,\cI}\le C(\rho,c,m)\norm{\cuu}_{1,\cI},$$
    where $C(\rho,c,m)>0$ is independent of $\calpha$.
    \end{proof}

    \begin{corollary}
   Let $\u\in\HH^{m+1}_{\alpha,\r}(\Ik)$, then  there is $C(S_{\pm},A_{\pm},m)>0$ independent of $\alpha$ and $h$ such that
    $$\left|\u- \Pi^{m}_{\calpha,\r} \u\right|_{i,\Ik}<C(S_{\pm},A_{\pm},m)h^{m+1-i}|\u|_{m+1,\Ik},\qquad 0\le i\le m.$$
    \end{corollary}
    \begin{proof}
    Let $\cuu=(\check{p},\check{u})^T=\M \u$ and let $\check{\boldsymbol{\pi}}^m_{\calpha,\r}\cu=(\check{\pi}^m_{\calpha,r^p} \check{p},\check{\pi}^m_{\calpha,r^u} \check{u})^T$, where $\check{\pi}^m_{\calpha,r^p} \check{p},\check{\pi}^m_{\calpha,r^u} $ are defined in \autoref{lem:hyper_orthogonality}. Then by \autoref{thm:immersed_bramble_hilbert}, we have
    \begin{equation}|\cuu-\check{\boldsymbol{\pi}}^m_{\calpha,\r}\cuu|_{i,\cI}\le C(\rho,c,m)|\cuu|_{m+1,\cI},\qquad i=0,1,\dots,m+1.
    \label{eqn:vector_BH}\end{equation}
    On the other hand, we have
    \begin{align}
        \left|\u- \Pi^{m}_{\calpha,\r} \u\right|_{i,\Ik}&=h^{1-i}\left|\cuu- \cPi \cuu\right|_{i,\cI}\notag \\
       &\le h^{1-i}\left(\left|\cPi\left(\cuu- \check{\boldsymbol{\pi}}^m_{\calpha,\r} \cuu\right)\right|_{i,\cI}
        +\left|\cuu- \check{\boldsymbol{\pi}}^m_{\calpha,\r}\cuu \right|_{i,\cI}
        \right)\notag \\
        &\le  C_1(\rho,c,m)h^{1-i}\left(\left|\cPi\left(\cuu- \check{\boldsymbol{\pi}}^m_{\calpha,\r} \cuu\right)\right|_{0,\cI}
        +\left|\cuu- \check{\boldsymbol{\pi}}^m_{\calpha,\r}\cuu \right|_{i,\cI}
        \right)\tag{Using \autoref{lemma:inverse_inequality}} \\
       & \le
       C_2(\rho,c,m)h^{1-i}\left(\norm{\cuu- \check{\boldsymbol{\pi}}^m_{\calpha,\r} \cuu}_{1,\cI}
        +\left|\cuu- \check{\boldsymbol{\pi}}^m_{\calpha,\r}\cuu \right|_{i,\cI}
        \right)\tag{Using \autoref{thm:Radau_stability}} \\
        & \le  C_3(\rho,c,m)h^{1-i}\norm{\cuu-\check{\boldsymbol{\pi}}^m_{\calpha,\r}\cuu }_{m+1,\cI}\notag \\
        &\le C_4(\rho,c,m)h^{1-i}|\cuu|_{m+1,\cI}\tag{From \eqref{eqn:vector_BH} }\\
        &= C_4(\rho,c,m)h^{m+1-i}|\u|_{m+1,\Ik}.\notag
    \end{align}

    \end{proof}

    By summing over all elements, we get a similar bound for the global Radau projection $\Ru$  with a function $\u\in \HH^{m+1}_{\alpha,\r}(I)$
    \begin{equation}
        \norm{\u- \mathcal{R}\u}_{i,I}<C(S_{\pm},A_{\pm},m)h^{m+1-i}|\u|_{m+1,I},\qquad 0\le i\le m.
        \label{eqn:Radau_global_rate}
    \end{equation}

\begin{theorem}
Let $\u$ be the solution of problem \eqref{eqn:acoustic_matrix_form} and let $\u_h\in\WW^m_{\alpha,\r}(I) $  be the solution of \eqref{eqn:general_IDG_form_compact}. If $\u \in C([0,T];\HH^{m+2}_{\alpha,\r}(I))$, then there is $C>0$ independent of $h$ and $\alpha$ such that
$$\norm{\u(\cdot,T)-\u_h(\cdot,T)}_{0,I}\le C h^{m+1}\left(|\u_0|_{m+1,I} + |\u(\cdot,T)|_{m+1,I}+T\max_{0\le t\le T}|\u(\cdot,t)|_{m+2,I}\right),\qquad T>0.$$

\end{theorem}

\begin{proof}
Our proof follows the usual methodology used for the non-interface problem (see \cite{cockburnIntroductionDiscontinuousGalerkin1998}). We first note that $\mathcal{R}\u_t=\frac{d}{dt}\mathcal{R}\u$ and split the error $\e=\u_h-\u$ as
$$\e=\z -\g, \qquad \z=\u-\mathcal{R}\u, \quad \g=\u_h-\mathcal{R}\u.
$$
It follows from the definition of $\mathcal{R}$ in \eqref{eqn:Global_Radau_projection} that

\begin{equation}B(\g,\g)= B(\g-\z,\g) =B(\u_h-\u,\g)=B(\e,\g).
\label{eqn:eg_equation}
\end{equation}
By combining \eqref{eqn:eg_equation}, \eqref{eqn:general_IDG_form_compact} and \eqref{eqn:negative_definite_form}, we get

\begin{align}
    \left(S\z_t(\cdot,t),\g(\cdot,t)\right)_{I}&=\left(S\g_t(\cdot,t),\g(\cdot,t)\right)_{I}-\left(S\e_t(\cdot,t),\g(\cdot,t)\right)_{I},\notag \\
    &=\frac{1}{2}\frac{d}{dt}\norm{\sqrt{S}\g(\cdot,t)}_{0,I}^2 -B(\e(\cdot,t),\g(\cdot,t)),\notag\\
    &=\frac{1}{2}\frac{d}{dt}\norm{\sqrt{S}\g(\cdot,t)}_{0,I}^2
    -B(\g(\cdot,t),\g(\cdot,t)),\notag\\
    &=\frac{1}{2}\frac{d}{dt}\norm{\sqrt{S}\g(\cdot,t)}_{0,I}^2+\sigma(t),\label{eqn:gq_inequality}
\end{align}
where  $\sigma(t)\ge 0$ by \eqref{eqn:negative_definite_form}. Let $\kappa(t)= \norm{\sqrt{S}\g(\cdot,t)}_{0,I}$, then by Cauchy-Schwarz inequality,
\begin{equation}\left(S\z_t(\cdot,t),\g(\cdot,t)\right)_{I}\le \norm{\z_t(\cdot,t)}_{0,I}\kappa(t) .\label{eqn:Young_ineq}\end{equation}
Following the ideas of the proof of \autoref{lem:commutative_diagram}, we can show that  $\u_t(\cdot,t)=-A\u_x(\cdot,t)\in\HH^{m+1}_{\alpha,\r}(I)$ since $\u(\cdot,t)\in\HH^{m+2}_{\alpha,\r}(I)$. Therefore, by \eqref{eqn:Radau_global_rate}, there is $C$ independent of $h$ and $\alpha$ such that
\begin{equation}
    \norm{\z_t(\cdot,t)}_{0,I}\le C h^{m+1}|\u(\cdot,t)|_{m+2,I}. \label{eqn:bound_q}
\end{equation}
Now, we use \eqref{eqn:bound_q} , \eqref{eqn:Young_ineq} and integrate  \eqref{eqn:gq_inequality} on $[0,T]$ to get
\begin{equation}
    \frac{1}{2}\kappa(T)^2-\frac{1}{2}\kappa(0)^2+\sigma(t)\le C h^{m+1}\int_0^T \kappa(s)|\u(\cdot,s)|_{m+2,I}\ ds,\label{eqn:Gronwall_form}
\end{equation}
Using a generalized version of Gronwall's inequality (see \cite[p. 24]{barbuDifferentialEquations2016}), we get the following bound on $\kappa(T)$

\begin{align}
    \kappa(T)&\le \kappa(0) +Ch^{m+1}\int_0^T |\u(\cdot,s)|_{m+2,I} \ ds,\\
    &\le \kappa(0) +Ch^{m+1}T\max_{0\le t\le T}|\u(\cdot,t)|_{m+2,I}. \label{eqn:Gronwall}
\end{align}
We also have
\begin{equation}
    \kappa(0)=\norm{\sqrt{S}\left(\u_h(\cdot,0)-\Ru_0\right)}_{0,I}\le \norm{\sqrt{S}\left(\u_h(\cdot,0)-\u_0\right)}_{0,I}+\norm{\sqrt{S}\left(\u_0-\Ru_0\right)}_{0,I}\le Ch^{m+1}|\u_0|_{m+1,I}.
    \label{eqn:h0_bound}
\end{equation}
We substitute \eqref{eqn:h0_bound} into \eqref{eqn:Gronwall}, to obtain

$$\kappa(T)= \norm{\sqrt{S}\g(\cdot,T)}_{0,I}\le Ch^{m+1}\left(|\u_0|_{m+1,I}+T\max_{0\le t\le T}|\u(\cdot,t)|_{m+2,I}\right).$$
To finalize the proof, we use the triangle inequality

$$\norm{\e(\cdot,T)}_{0,I}\le \norm{\z(\cdot,T)}_{0,I}+\norm{\g(\cdot,T)}_{0,I}\le Ch^{m+1}\left(|\u_0|_{m+1,I} + |\u(\cdot,T)|_{m+1,I}+T\max_{0\le t\le T}|\u(\cdot,t)|_{m+2,I}\right).$$
\end{proof}

 \section{Novel proofs for results already established in the literature}
 \label{sec:Already_established_results}

In this section, for demonstrating the versatility of the immersed scaling argument established in  \autoref{sec:properties_of_v} and \autoref{sec:Bramble_Hilbert}, we redo the error estimation for two IFE methods in the literature. One of them is the IFE space for an elliptic interface problem
\cite{adjeridPthDegreeImmersed2009}, and the other one is the IFE space for an interface problem of the Euler-Bernoulli Beam
\cite{linErrorAnalysisImmersed2017}. We note that the approximation capability for these IFE spaces were already analyzed, but with complex and lengthy procedures. Our discussions here is to demonstrate that similar error bounds for the optimal approximation capability of these different types of IFE spaces
can be readily derived by the unified immersed scaling argument.

\commentout{
    we use the results of \autoref{sec:properties_of_v} and \autoref{sec:Bramble_Hilbert} to re-derive some error estimate that exist already in the literature. We are mainly interested in the elliptic interface problem \cite{adjeridPthDegreeImmersed2009} $$-\beta u''=f,\qquad \beta \text{ is piecewise constant}.$$
    As well as the Euler-Bernoulli Beam interface problem \cite{linErrorAnalysisImmersed2017}
    $$\beta u^{(4)}=f,\qquad \beta \text{ is piecewise constant}.$$

Our goal here is to provide a unified framework to proving optimal convergence estimates. Although some estimates exist in the literature, the proofs usually rely on multi-point Taylor expansions and the evaluation of nested integrals. Our approach, on the other hand,  is shorter, simpler and has potential to be applied to higher dimensions.

In \autoref{subsec:continuous_FE}, we analyze the IFE method for the elliptic interface problem and prove optimal error estimates for all degrees $m$. In \autoref{subsec:euler_bernouli}, following the authors of \cite{linErrorAnalysisImmersed2017}, we restrict our attention only to the lowest order (cubic) conforming IFE method for the Euler-Bernoulli Beam interface problem. We believe that our approach could be applied smoothly to higher degrees.
}

 \subsection{The $m$-th degree IFE space for an elliptic interface problem}\label{subsec:continuous_FE}
In this subsection, we consider the $m$-th degree IFE space developed in \cite{adjeridPthDegreeImmersed2009} for solving the following interface problem:
\begin{equation}
\begin{cases}-\beta(x)u''(x)=f(x),\ x\in(a,\alpha)\cup(\alpha,b) \\
u(a)=u(b)=0,
\end{cases}\qquad \beta(x)=\begin{cases}\beta^->0,& x\in (a,\alpha),\\
\beta^+>0,& x\in (\alpha,b),\end{cases}\qquad [u]_{\alpha}=[\beta u']_{\alpha}=0.\label{eqn:elliptic_problem_statement}
\end{equation}
Assume that $f$ is in $C^{m-1}(I)$ which implies that the solution $u\in \H^{m+1}_{\alpha,r}(I)$ with
\begin{equation}
r_0=1,\quad \text{ and }\quad   r_i=\frac{\beta^-}{\beta^+}\ \text{ for }\ i=1,2,\dots,m.\label{eqn:r_beta_connection}
\end{equation}
The discussion  in \autoref{sec:notation} suggests the following IFE space for this elliptic interface problem:
\begin{eqnarray}
Z_{\alpha,r}^m(\mathcal{T}_h)= H^{1}_0(I)\cap W_{\alpha,r}^m(\mathcal{T}_h) \label{eq:IFE_Space_elliptic}
\end{eqnarray}
which coincides with the one developed in \cite{adjeridPthDegreeImmersed2009} based on the extended jump conditions where
it was proved, by an elementary but complicated multi-point Taylor expansion technique, to have the optimal approximation capability with respect to $m$-th degree polynomials employed in this IFE space. We now reanalyze this IFE space by the
immersed scaling argument.

\commentout{
        we will show that the results of \autoref{sec:properties_of_v} and \autoref{sec:Bramble_Hilbert} still hold for the continuous subspace:
        $$Z_{\alpha,r}^m(\mathcal{T}_h)= H^{1}_0(I)\cap W_{\alpha,r}^m(\mathcal{T}_h),\qquad \text{assuming that } r_0=1,$$
        where $C^{0}(I)$ is the space of scalar continuous functions on $I$ and $W_{\alpha,r}^m(\mathcal{T}_h)$ is defined in \eqref{eqn:global_IFE_space}. The assumption $r_0=1$ is essential to have a non-empty intersection. The space $Z_{\alpha,r}^m(\mathcal{T}_h)$ appears in \cite{adjeridPthDegreeImmersed2009} as a conforming IFE space for the second order elliptic problem with discontinuous coefficients:

        The first immersed finite element method was developed by Li \cite{liImmersedInterfaceMethod1998} to solve \eqref{eqn:elliptic_problem_statement} with piecewise linear basis functions and was extended to higher order basis functions in \cite{adjeridPthDegreeImmersed2009}. Although the problem \eqref{eqn:elliptic_problem_statement} is one-dimensional,  the properties of the IFE approximation are still investigated, see for instance \cite{caoSuperconvergenceImmersedFinite2017} for an analysis of the superconvergence properties of the IFE method. In this section, we will provide a new and simple approach for the error analysis of the IFE method based on the scaling argument and the immersed Bramble-Hilbert lemma detailed in \autoref{sec:Bramble_Hilbert} instead of the multi-point Taylor expansion used in \cite{adjeridPthDegreeImmersed2009}. In the rest of this paper, we assume that the right hand side $f$ is in $C^{m-1}(I)$ which implies that the solution $u\in \H^{m}_{\alpha,r}(I)$, where \begin{equation} r_0=1,\quad \text{ and }\quad   r_i=\frac{\beta^-}{\beta^+}\ \text{ for }\ i=1,2,\dots,m.\label{eqn:r_beta_connection}\end{equation}
        However, the results that we obtain here hold for  all sequences $\{r_i\}_{i=0}^m$ with $r_0=1$. Since the conditions of the Aubin-Nitsche trick are satisfied for the problem \eqref{eqn:elliptic_problem_statement}, it is enough to study the approximation capabilities of $Z^m_{\alpha,r}(\T_h)$ as a finite dimensional of $\H^{m}_{\alpha,r}(I)$. Moreover, we are only interested in Lobatto projections $P_h:\H^{m}_{\alpha,r}(\Ik)\to \V^{m}_{\alpha,r}(\Ik)$ that satisfy $P_hu(x_{k_0-1})=u(x_{k_0-1})$ and $P_hu(x_{k_0})=u(x_{k_0})$. The interpolation operator is an example of a such projection, we have shown that the interpolation exists in \autoref{thm:Lagrange basis}.
}
The continuity of functions in the IFE space suggests to consider the following immersed Lobatto projection  $\L^m_{\alpha,r}:\H^{m+1}_{\alpha,r}(\Ik)\to\V^m_{\alpha,r}(\Ik)$  defined by
\begin{equation}
    \begin{cases}
    \L^m_{\alpha,r} u(x_{k_0-1})=u(x_{k_0-1}),\\
    \L^m_{\alpha,r} u(x_{k_0})=u(x_{k_0}),\\
    \left( \L^m_{\alpha,r} u,v_h\right)_{w,\Ik}=\left( u,v_h\right)_{w,\Ik}, \quad \forall v_h\in \V^{m-2}_{\alpha,\tau^2(r)}(\Ik),
    \end{cases}\qquad w(x)=\begin{cases}
    r_1,& x\in\Ik^-,\\
    1, &x\in \Ik^+,
    \end{cases}
    \label{eqn:Lobatto_projection}
\end{equation}
where $\tau^2=\tau\circ\tau$ and $\tau$ is the  shift operator defined in \eqref{eqn:first_shift_operator}. The related reference immersed Lobatto projection $\cL:\H^{m+1}_{\calpha,r}(\cI)\to \cV$ is defined by the diagram \eqref{eqn:phys_ref_diagram}, that is, $\cL\check{u}=\L^m_{\alpha,r} u$ where $\check{u}=\mathcal{M}v$.

% \vspace{3cm}

% \begin{equation}
%     \begin{cases}
%     \ds \cL u(0)=u(0),\\
%     \ds \cL u(1)=u(1),\\
%     \ds \left( \cL u,v_h\right)_{\cw,\cI}=\left( u,v_h\right)_{\cw,\cI}\quad \forall v\in \V^{m-2}_{\calpha,\tau^2(r)}(\cI),
%     \end{cases}\qquad \cw(x)=\begin{cases}
%     r_1,& x\in\cI^-,\\
%     1, &x\in \cI^+.
%     \end{cases}
%     \label{eqn:reference_Lobatto_projection}
% \end{equation}

 For simplicity, let $\ucl=\cL \check{u}$ for a given $\check{u}\in \cV$ and note that the system \eqref{eqn:Lobatto_projection} is a square system of $m+1$ equations since the last line can be written as $m-1$ equations.
 Therefore, we only need to show that if $\check{u}\equiv 0$ then $\ucl\equiv 0$ to prove that $\cL$ is well defined.

\begin{lemma}
\label{lem:well_posdeness_of_lobatto}
The reference immersed Lobatto projection $\cL$ is well defined.
\end{lemma}

\begin{proof}
Let $\check{u}\equiv 0$, we will show that $\ucl=\cL \check{u}\equiv 0$. We have
$$\ucl(0)=\ucl(1)=0,\qquad \left( \ucl,v_h\right)_{\cw,\cI}=0,\ \forall v\in \V^{m-2}_{\calpha,\tau^2(r)}(\cI), $$
where $\check{w}=\mathcal{M}w$.
Using \eqref{eqn:first_shift_operator},  $\ucl''\in \V^{m-2}_{\calpha,\tau^2(r)}(\cI)$, then

\begin{align*}
0=\int_0^1 w(x)\ucl(x)\ucl''(x)\ dx&=
r_1\int_0^{\calpha} \ucl(x)\ucl''(x)\ dx+\int_{\calpha}^1 \ucl(x)\ucl''(x)\ dx
\\
&=\ucl(\calpha)\left[r_1\ucl'(\calpha^-)-\ucl'(\calpha^+)\right] -\int_0^1 w(x)[\ucl'(x)]^2\ dx\\
0&=-\int_0^1 w(x)[\ucl'(x)]^2\ dx,
\end{align*}
which implies that $\ucl$ is zero since $\ucl(0)=\ucl(1)=0$.
\end{proof}

Next, we will show that $\{\cL\}_{0\le \calpha<1}$ is a uniformly bounded collection of RIFE projections in the following lemma.

\begin{lemma}\label{lem:Lobatto_is_uniformly_bounded}
There is a constant $C(\beta^+,\beta^-,m)>0$ independent of $\calpha$ such that the following estimate holds for
every $\check{u}\in \H^{m+1}_{\calpha,r}(\cI)$
$$\norm{\ucl}_{0,\cI} \le C(\beta^+,\beta^-,m)\norm{\check{u}}_{1,\cI}.$$
\end{lemma}

\begin{proof} We  write $\ucl$ as  $\ucl=q_1+q_2$, where $q_1\in \V^{1}_{\calpha,r}(\cI)$  such that $$q_1(0)=\check{u}(0),\quad q_1(1)=\check{u}(1),$$
and $q_2=\cL(\check{u}-q_1)\in \V^m_{\calpha,r}(\cI)$. The construction of $q_1$ is straightforward (see \cite{liImmersedInterfaceMethod1998}) and we have $\norm{q_1}_{0,\cI}\le C(\beta^+,\beta^-)\norm{u}_{1,\cI}$. Now, the second term $q_2$ satisfies
$$q_2(0)=q_2(1)=0,\qquad \left( q_2,v_h\right)_{\cw,\cI}=\left(\check{u}-q_1,v_h\right)_{\cw,\cI},\ \forall v_h\in \V^{m-2}_{\calpha,\tau^2(r)}(\cI), $$
where $\check{w}=\mathcal{M}w$.
Following the proof of \autoref{lem:well_posdeness_of_lobatto}, we can choose $v_h=q_2''$ and integrate by parts to get
$$-\norm{wq_2'}_{0,\cI}^2 = \left(\check{u}-q_1,q_2''\right)_{\cw,\cI}.$$
We take the absolute value of each side and apply Cauchy-Schwarz inequality
$$\norm{q_2'}_{0,\cI}^2 \le  C(\beta^+,\beta^-,m)\norm{\tilde{u}-q_1}_{0,\cI}\norm{q_2''}_{0,\cI}.
$$
The inverse inequality in \autoref{lemma:inverse_inequality} implies that $\norm{q_2''}_{0,\cI}\le C(\beta^+,\beta^-,m)\norm{q_2'}_{0,\cI}$. Hence,
\begin{equation}\norm{q_2'}_{0,\cI}\le  C(\beta^+,\beta^-,m)\norm{\check{u}-q_1}_{0,\cI}\le C(\beta^+,\beta^-,m)\norm{\check{u}}_{1,\cI}.\label{eqn:lobatto_q2prime}\end{equation}
Since $q_2(0)=q_2(1)=0$, we can apply Poincaré's inequality to obtain $\norm{q_2}_{0,\cI}<C \norm{q_2'}_{0,\cI}$. Finally, we have

$$\norm{\tilde{u}}_{0,\cI}\le \norm{q_1}_{0,\cI}+\norm{q_2}_{0,\cI}\le C(m,\beta^+,\beta^-)\norm{\check{u}}_{1,\cI}.$$
\end{proof}
Then, we can use \autoref{thm:general_theorem_LIFE} to derive an error bound for the Lobatto projection
$\L^m_{\alpha,r} u$ in the following theorem which confirms the optimal approximation capability of the
IFE space established in \cite{adjeridPthDegreeImmersed2009} by a more complex analysis.

\begin{theorem}\label{thm:H1_local_approximation}
There is $C(\beta^+,\beta^-,m)>0$ such that the following estimate holds for every ${u}\in \H^{m+1}_{\alpha,r}(\Ik)$
$$|u-\L^m_{\alpha,r} u|_{i,\Ik}\le C(\beta^+,\beta^-,m)h^{m+1-i}|u|_{m+1,\Ik},\quad \forall i=0,1,\dots,m.
$$
\end{theorem}

\begin{proof}
    This follows immediately from \autoref{lem:Lobatto_is_uniformly_bounded} and \autoref{thm:general_theorem_LIFE}.
\end{proof}

\commentout{
        Finally, we can apply Céa's lemma and the Aubin-Nitsche trick to obtain an $L^2$ estimate on the error of the IFE solution:

        \begin{theorem}\label{thm:IFE_elliptic_error}
        Let $u\in \H^m_{\alpha,r}(I)$ the  solution to \eqref{eqn:elliptic_problem_statement} and let $u_h\in Z^m_{\alpha,r}(I)$ be the solution to the discrete problem
        $$\left(\beta u_h',v_h'\right)_{I}= (f,v_h)_{I},\qquad \forall v_h\in Z^m_{\alpha,r}(I),$$
        then
        $$\norm{u-u_h}_{0,\cI}\le Ch^{m+1}|u|_{m+1,I},$$
        where $C$ is independent of $u$, $h$ and $\alpha$.
        \end{theorem}

        The result of \autoref{coro:H1_local_approximation} can be applied to the heat equation with discontinuous coefficients:

        \begin{equation}\begin{cases}u_t(x,t)=\beta(x)u_{xx}(x,t),\ x\in(a,b) \\
        u(a)=u(b)=0,\\
        u(x,0)=u_0(x),
        \end{cases}\qquad \beta(x)=\begin{cases}\beta^->0,& x\in (a,\alpha),\\
        \beta^+>0,& x\in (\alpha,b),\end{cases}\qquad [u]_{\alpha}=[\beta u_{x}]_{\alpha}=0.\label{eqn:parabolic_problem_statement}\end{equation}
        We first notice that by differentiating the interface conditions with respect to $t$, we obtain the following extended jump conditions:
        $$\frac{\partial^ku}{\partial x^k}(\alpha^+,t)= r_k\frac{\partial^ku}{\partial x^k}(\alpha^-,t),\qquad r_k=\left(\frac{\beta^-}{\beta^+}\right)^{\left\lceil \frac{k}{2}\right\rceil}, $$
        where $\lceil\cdot\rceil$ is the ceiling function.
}

\subsection{Euler-Bernoulli Beam interface problem} \label{subsec:euler_bernouli}

In this subsection, we apply the immersed scaling argument to reanalyze the cubic IFE space developed in
\cite{linImmersedFiniteElement2011} and \cite{wangHermiteCubicImmersed2005} for solving the following interface problem
of the Euler-Bernoulli beam equation:
\begin{equation}
\begin{cases}\beta(x)u^{(4)}(x)=f(x),\ x\in(a,\alpha)\cup(\alpha,b) \\
u(a)=u(b)=0,\\
u'(a)=u'(b)=0
\end{cases}\quad \beta(x)=\begin{cases}\beta^->0,& x\in (a,\alpha),\\
\beta^+>0,& x\in (\alpha,b),\end{cases}
\label{eqn:euler_bernoulli_problem_statement}
\end{equation}
where the solution $u$ satisfies the following jump conditions at $\alpha$
$$[u]_{\alpha}=[u']_{\alpha}=[\beta u'']_{\alpha}=[\beta u''']_{\alpha}=0.$$
First, let $r=\left(1,1,\frac{\beta^-}{\beta^+},\frac{\beta^-}{\beta^+}\right)$ be fixed throughout this subsection. Then, the
usual weak form of \eqref{eqn:euler_bernoulli_problem_statement} suggests to consider the following IFE method:
\begin{equation}\text{find } u_h\in Q^3_{\alpha,r}(I)\ \text{ such that }\ \left(\beta u_h'',v_h''\right)_{I}=(f,v_h)_{I},\ \forall v_h\in Q^3_{\alpha,r}(\T_h),
\label{eqn:Euler_bernoulli_discrete}
\end{equation}
where $Q^3_{\alpha,r}(\T_h)=H^{2}_0(I)\cap W^3_{\alpha,r}(\T_h)$. We note that the IFE space $Q^3_{\alpha,r}(\T_h)$ as well as the method described by \eqref{eqn:Euler_bernoulli_discrete} were discussed in \cite{linImmersedFiniteElement2011} and \cite{wangHermiteCubicImmersed2005}, and an error analysis based on a multipoint Taylor expansion was carried out to establish the optimality of this IFE method in \cite{linErrorAnalysisImmersed2017}. As another demonstration of the versatility of the immerse scaling argument, we now present an alternative analysis for the optimal approximation capability of this IFE space. This new analysis based on the framework developed in \autoref{sec:properties_of_v} and \autoref{sec:Bramble_Hilbert} is shorter and cleaner than the one in the literature.

\commentout{
    In , cubic Hermite IFE spaces were developed to solve \eqref{eqn:euler_bernoulli_problem_statement} and optimal convergence rate was reported. This was later proved in \cite{linErrorAnalysisImmersed2017}. We will provide a shorter proof here using \autoref{sec:properties_of_v} and \autoref{sec:Bramble_Hilbert}.
}

As usual, for the discussion of the approximation capability of the IFE space, we consider the interpolation on the reference element $\cI$ and map it to the physical element $\Ik$. To define the interpolation, we let $\{\sigma_i\}_{i=1}^4$ be the Hermite degrees of freedom, that is,
\begin{eqnarray*}
\sigma_0(v)=v(0),\quad \sigma_1(v)=v(1),\quad \sigma_2(v)=v'(0),\quad \sigma_3(v)=v'(1),\qquad \forall v\in H^2(\cI).
\end{eqnarray*}
It is known \cite{linImmersedFiniteElement2011,wangHermiteCubicImmersed2005} that there is a basis $\{L^i_{\calpha,r}\}_{i=0}^3$ of $\V^3_{\calpha,r}(\cI)$ that satisfies
\begin{equation}
\sigma_i( L^j_{\calpha,r})=\delta_{i,j},\qquad i,j=0,1,2,3. \label{eqn:Hermite_basis}
\end{equation}
These basis functions can then be used to define an immersed Hermite projection/interpolation operator
$\check{\S}_{\calpha,r}: \H^{4}_{\calpha,r}(\cI) \rightarrow \V^{3}_{\calpha,r}(\cI)$ such that $\check{u}_H=\check{\S}_{\calpha,r}\check{u}$ and
\begin{equation}
\check{u}_H=\sum_{i=0}^3 \sigma_i(\check{u})L^{i}_{\calpha,r}.\label{eqn:IHP_dof_definition}
\end{equation}

\commentout{
        Now, let $I^h_{\alpha,r}:\H^4_{\alpha,r}(I)\to Q^3_{\alpha,r}(\T_h)$ be the second order interpolation operator: $$(I^h_{\alpha,r}u)(x_k)=u(x_k),\quad (I^h_{\alpha,r}u)'(x_k)=u'(x_k)\qquad k=0,1,\dots, N.$$
        The construction of $I^h_{\alpha,r}u$ on non-interface elements is straightforward. Here, we focus on the construction of $I^h_{\alpha,r}u$ on the interface element.
        Let $\check{\S}_{\calpha,r} $ be the reference  immersed Hermite projection that maps $\cu \in \H^{4}_{\calpha,r}(\cI)$ to $\cu_H\in\V^{3}_{\calpha,r}(\cI)$, where $$\sigma_i(\cu_H)=\sigma_i(\cu),\qquad i=0,1,2,3.$$
        We first show that $\check{\S}_{\calpha,r}$ is well defined.

        \begin{lemma}
            \label{lem:Hermite_defined}
            Let $\beta^{\pm}>0$, $\calpha\in (0,1)$ and $\cu\in  \H^{4}_{\calpha,r}(\cI)$, then there is a unique $\cu_H\in \V^{4}_{\calpha,r}(\cI)$ such that
        \begin{equation}\sigma_i(\cu_H)=\sigma_i(\cu),\qquad i=0,1,2,3.\label{eqn:sigma_def_u_H}\end{equation}
        \end{lemma}

        \begin{proof}
            It is enough to show that if $\cu\equiv 0$, then $\cu_H$ defined in \eqref{eqn:sigma_def_u_H} is identically zero.  If $u\equiv 0$, then $\cu_H(0)=\cu_H(1)=0$, by Rolle's theorem, there is $c\in(0,1)$ such that $\cu_H'(c)=0$. Hence, $\cu_H'$ has three distinct zeros since $\cu_H'(0)=\cu_H'(1)=0$. Using \autoref{thm:Zeros of IFE basis functions}, we conclude that $\cu_H'\equiv 0$ since $\cu_H'\in \V^{2}_{\calpha,r}(\cI)$. Consequently, $\cu_H\equiv 0$.
        \end{proof}

}

\begin{lemma}\label{lem:bounds_on_Hermite}
    Let $\beta^\pm>0$ and $\calpha \in(0,1)$, then
\begin{align}
-1<L^i_{\calpha,r}(x)<1,\qquad \forall\ x\in[0,1],\quad i=0,1,2,3. \label{eq:bnds_for_beam_basis}
\end{align}
\end{lemma}

\begin{proof}
 See \autoref{sec:Hermite_proof}
\end{proof}
Now, we are ready to establish that $\{\check{\S}_{\calpha,r}\}_{0<\calpha<1}$ is a collection of uniformly bounded  of RIFE projections.

\begin{lemma}\label{lem:ihp_basis_bounded}
Let $\beta^\pm>0, \calpha\in(0,1)$. Then there is a constant $C$ independent of $\calpha$ such that the following estimate holds for every
$\cu\in \H^{4}_{\calpha,r}(\cI)$ 
$$\norm{\check{\S}_{\calpha,r}\check{u}}_{0,\cI}\le C\norm{\check{u}}_{2,\cI}.$$
\end{lemma}

\begin{proof}
    We know that $\sigma_i(\check{u})\le C \norm{\check{u}}_{2,\cI}$ since $\check{u}\in H^2(\cI)$. Now, we apply the triangle inequality to \eqref{eqn:IHP_dof_definition}  and \autoref{lem:bounds_on_Hermite} to get
    $$\norm{\check{\S}_{\calpha,r}\check{u}}_{0,\cI}\le C\norm{\check{u}}_{2,\cI} \left(\sum_{i=0}^3 \norm{L^i_{\calpha,r}}_{0,\cI}\right)\le 4C\norm{\check{u}}_{2,\cI}.$$
\end{proof}

Now, let $\S_{\alpha,r}=\M^{-1}\circ \S_{\calpha,r}\circ\M$ where $\M$ is defined in \eqref{eqn:translation}. By the commutative diagram in \eqref{eqn:phys_ref_diagram}, $\S_{\alpha,r}$ is the local immersed Hermite interpolation. Then, by \autoref{lem:ihp_basis_bounded}, $\{\S_{\alpha,r}\}_{\calpha\in \overset{\circ}{I}_{k_0}}$ is a collection of uniformly bounded LIFE projections. Hence, the following theorem follows from \autoref{thm:general_theorem_LIFE}.
\begin{theorem}
Let $\beta^\pm>0$, $i\in\{0,1,2,3\}$,  $\alpha\in\Ik$. Then, there is a constant $C$ independent of $\alpha$ such that 
the following estimate holds for every  $u\in \H^{4}_{\alpha,r}(\Ik)$
$$\norm{u-{\S}_{\alpha,r}u}_{i,\Ik}\le C h^{3-i}|u|_{4,\Ik}.$$
\end{theorem}
This theorem establishes the optimal approximation capability of the IFE space $Q^3_{\alpha,r}(\T_h)$ which was first derived in
\cite{linErrorAnalysisImmersed2017} with a lengthy and complex procedure.

\section{Conclusion}
In this manuscript  we developed a framework for analyzing the approximation properties of one-dimensional IFE spaces using the scaling argument. We have applied this IFE scaling argument to establish the optimal convergence of IFE spaces constructed for solving the acoustic interface problem, the elliptic interface problem and the Euler-Bernoulli beam interface problem, respectively. We are currently working on extending these results to IFE spaces and  methods for solving  interface problems in two and three dimensions.

\begin{appendices}
    \section{Proof of \autoref{lem:norm_on_Q}}
\label{sec:proof_of_lemma_4}
Our goal is to show that the ratio $\frac{\sqrt{\cp(0)^2+\cp(1)^2}}{\norm{\cp}_{0,\cI}}$ is bounded from below by a constant $c(m,r,w)$ independent of $\calpha$.
For simplicity, let $q_i\in\P^{m}([0,1])$ be the monomial basis $q_i(x)=x^i$ for $0\le i\le m$. Using the equivalence of norms, one can show that there is $c_1(m)>0$ such that
\begin{equation}\min\left(|p(0)|,|p(1)|\right)+\sum_{i=0}^{m-1}\left|(p,q_i)_{[0,1]}\right|\ge c_1(m)\norm{p}_{0,[0,1]},\qquad \forall\ p\ \in \P^{m}([0,1]).
\label{eqn:legendre_norm_equivalence}
\end{equation}

Unfortunately, if we extend \eqref{eqn:legendre_norm_equivalence} to $\V^{m}_{\calpha,r}$, then the constant on the right might depend on $\calpha$ and might grow unboundedly as $\calpha\to0^+$ or as $\calpha\to 1^{-}$. To circumvent this issue, we will use a scaling trick similar to the one used in the proof of \autoref{lem:norm_of_extension}. First, we bound $(\cp,q_i)_{w_s,\cI^s}$ as shown in the following lemma

\begin{lemma}
\label{lem:almost orthogonality}
Let $\tilde{m}\ge m\ge 0$, $\{r_{k}\}_{k=0}^m\subset \mathbb{R}_+$ and $\calpha\in (0,1)$, there is $C(m,r,w)>0$ such that if $\ch_s>\ch_{s'}$, then
$$\left|(\cp_s,q_i)_{\cI^s}\right|=\left|\int_{\cI^s}\cp_s(x)x^idx\right|\le C(m,r,w)h_{s'}\norm{\cp_s}_{0,\cI^s},\quad i=0,1,\dots,m-1,\ \forall\cp\in \cQ.$$
\end{lemma}
\begin{proof} Since $\cp\in \cQ$, we have $$0=w_s(\cp_s,q_i)_{\cI^s}+w_{s'}(\cp_{s'},\Esp(q_i))_{\cI^{s'}}.$$
Then, by Cauchy-Schwarz inequality and \eqref{eqn:norm_of_smallest}, we have

\begin{align*}\left|(\cp,q_i)_{\cI^s}\right|&=\frac{w_{s'}}{w_s}\left|(\cp,\Esp(q_i))_{\cI^{s'}}\right|\le
C(w)\norm{\cp}_{0,\cI^{s'}} \norm{\Esp(q_i)}_{0,\cI^{s'}},\\
&\le C(m,r,w)\sqrt{h_{s'}}\norm{\cp}_{0,\cI^s}\sqrt{h_{s'}}\norm{q_i}_{0,\cI^s},\\
&\le C(m,r,w)h_{s'}\norm{\cp}_{\cI^s}.
\end{align*}
\end{proof}
The previous lemma shows that $(\cp,q_i)_{\cI^s}$ will approach $0$ if $h_{s}$ approaches $1$. This will allow us to obtain a restricted version of \eqref{eqn:legendre_norm_equivalence}.
\begin{lemma}
\label{lem:delta_pointwise_bound}
There is $\delta(m,r,w)\in(0,\frac{1}{2})$ and $C(m,r)>0$ such that if $\min(h_{-},h_{+})<\delta(m,r,w)$, then
$$
|\cp(0)|+|\cp(1)|\ge C(m,r)\norm{\cp}_{0,\cI}. \qquad \forall \cp\in\cQ.
$$
\end{lemma}
\begin{proof} We will only discuss the case where $h_{-}>h_{+}$, the other case can be proved similarly. We define $\hpm\in \P^m([0,1])$  as $\hpm(\xi)=\cp_-(\ch_- \xi)$, then by
the fact that $h_- \geq 1/2$,
\begin{align*}
|\cp(0)|+\sum_{i=0}^{m-1}\left|\int_0^{h_-}\cp(x)x^i\ dx\right|&=|\hpm(0)|+\sum_{i=0}^{m-1}h_-^{i+1}\left|\int_0^{1}\hpm(\xi)\xi^i\ d\xi\right|\\
&\ge |\hpm(0)|+\sum_{i=0}^{m-1}2^{-m-1}\left|(\hpm,q_i)_{[0,1]}\right| \ge C(m)\left(|\hpm(0)|+\sum_{i=0}^{m-1}\left|(\hpm,q_i)_{[0,1]}\right|\right).\\
\end{align*}
Then, by \eqref{eqn:legendre_norm_equivalence}, $h_- \leq 1$, and \eqref{eqn:norm_of_smallest}, we have
\begin{align*}
|\cp(0)|+\sum_{i=0}^{m-1}\left|\int_0^{h_-}\cp(x)x^i\ dx\right| &\geq
C(m) \norm{\hpm}_{0,[0,1]} =C(m) h_-^{-1/2}\norm{\cp}_{0,\cI^-} \ge C(m) \norm{\cp}_{0,\cI^-}\ge C_0(m,r)\norm{\cp}_{0,\cI}.
\end{align*}
Now, we use \autoref{lem:almost orthogonality} to  estimate the inner product on the left hand side:
$$\sum_{i=0}^{m-1}\left|\int_0^{h_-}\cp(x)x^i\ dx\right|\le C_1(m,r,w)h_{+} \norm{\cp}_{0,\cI^-}.$$
We combine it with the previous inequality to get
$$|\cp(0)|\ge\norm{\cp}_{0,\cI} \left(C_0(m,r)-C_1(m,r,w)h_+\right).$$
Hence, if $h_+\le \delta= \min(1,\frac{C_0(m,r)}{2C_1(m,r,w)})$, then
\begin{equation}
|\cp(0)|\ge\frac{1}{2}C_0(m,r)\norm{\cp}_{0,\cI^-}
\label{eqn:cp0_lower_bound}
\end{equation}
A similar argument can be used to show that if $h_+\le \tilde{\delta}$ (where $\tilde{\delta}$ could be different than the previous $\delta$), then
\begin{equation}
|\cp(1)|\ge\frac{1}{2}\tilde{C}_0(m,r)\norm{\cp}_{0,\cI^+}.
\label{eqn:cp1_lower_bound}
\end{equation}
\end{proof}

So far, we have shown that if one of the sub-elements $\cI^{\pm}$ is small enough, then \autoref{lem:norm_on_Q} holds. It remains to show that the lemma holds for $\calpha\in[\delta,1-\delta]$, for which, we consider the following sequence $\{\mathcal{O}^{i}_{\calpha,w,r}\}_{i=0}^m$ by the  Gram-Schmidt process:
\begin{equation}\mathcal{O}^{0}_{\calpha,w,r}=\N^0_{\calpha,r},\quad
\mathcal{O}^{i}_{\calpha,w,r}=\N^i_{\calpha,r}-\sum_{j=0}^{i-1}\frac{(\N^i_{\calpha,r},\mathcal{O}^{j}_{\calpha,w,r})_{w,\cI}}{(\mathcal{O}^{j}_{\calpha,w,r},\mathcal{O}^{j}_{\calpha,w,r})_{w,\cI}}\mathcal{O}^{j}_{\calpha,w,r},\quad i=1,2,\dots,m.\label{eqn:Construction_of_O}\end{equation}
Clearly, we have $\cO\in \cQ$. The following lemma shows that when $\cO$ is expressed in terms of the canonical basis $\{\N^i_{\calpha,r}\}_{i=0}^m$, the coefficients of the expansion are rational functions.
%This will be useful later to show that $\mathcal{J}_{w,r}^m$ is continuous.

\begin{lemma}
\label{lem:expansion_of_O}
Let $\tilde{m}\ge m\ge 0$, $\{r_{k}\}_{k=0}^m\subset \mathbb{R}_+$ and $\calpha\in (0,1)$, the orthogonal RIFE function $\cO$ defined in \eqref{eqn:Construction_of_O} satisfies
\begin{equation}
\cO=\sum_{i=0}^m R_{w,r}^{i,m}(\calpha)\N^i_{\calpha,r} \label{eqn:expansion_of_O}\end{equation}
for some rational functions $\left\{R_{w,r}^{i,m}\right\}_{i=0}^m$ of $\calpha$.
\end{lemma}

\begin{proof}
We will prove \autoref{lem:expansion_of_O} via strong induction. First, the case $m=0$ is obvious. Now,we assume that For every $i=0,1,\dots,m-1$, there are rational functions $R^{j,i}_{w,r}$ of $\calpha$ such that
\begin{equation}\mathcal{O}^i_{\calpha,w,r} =\sum_{j=0}^i R_{w,r}^{j,i}(\calpha)\N^j_{\calpha,r}.
\label{eqn:Oi_expansion}
\end{equation}
To show that $\cO$ satisfies \eqref{eqn:expansion_of_O}, we use the fact that $(\N^i_{\calpha,r},\N^j_{\calpha,r})_{w,\cI}$ is a polynomial in $\calpha$. Therefore,
\begin{equation}
\frac{(\N^i_{\calpha,r},\mathcal{O}^{j}_{\calpha,w,r})_{w,\cI}}{(\mathcal{O}^{j}_{\calpha,w,r},\mathcal{O}^{j}_{\calpha,w,r})_{w,\cI}}
\label{eqn:inner_prod_GM}
\end{equation}
is a rational function of $\calpha$. Furthermore, by plugging \eqref{eqn:Oi_expansion} into \eqref{eqn:Construction_of_O} and rearranging the terms, we get

$$\mathcal{O}^m_{\calpha,w,r}=\N^m_{\calpha,r}-\sum_{j=0}^{m-1}
\frac{(\N^m_{\calpha,r},\mathcal{O}^{j}_{\calpha,w,r})_{w,\cI}}{(\mathcal{O}^{j}_{\calpha,w,r},\mathcal{O}^{j}_{\calpha,w,r})_{w,\cI}}
\sum_{i=0}^j R^{i,j}_{w,r}(\calpha)\N^i_{\calpha,r} =
\N^m_{\calpha,r}-\sum_{i=0}^{m-1}\underbrace{\left(\sum_{j=i}^{m-1}
\frac{(\N^m_{\calpha,r},\mathcal{O}^{j}_{\calpha,w,r})_{w,\cI}}{(\mathcal{O}^{j}_{\calpha,w,r},\mathcal{O}^{j}_{\calpha,w,r})_{w,\cI}}
 R^{i,j}_{w,r}(\calpha)\right)}_{:=R^{j,m}_{w,r}(\calpha)}\N^i_{\calpha,r}.
$$
From the strong induction assumption and \eqref{eqn:inner_prod_GM}, we conclude that $R^{j,m}_{w,r}$ is a rational function.
\end{proof}

\begin{corollary}\label{coro:J_is_rational}
Given $m,w_{\pm}$ and $r$, the function $\mathcal{J}_{w,r}^m:(0,1)\to \mathbb{R}_+$ defined in \eqref{eqn:definition_of_J} is a rational function.
\end{corollary}

Now, we are ready to prove \autoref{lem:norm_on_Q}. We can rewrite \autoref{lem:delta_pointwise_bound} as: there is $\delta\in (0,\frac{1}{2})$ that depends on $m,w$ and $r$, and a constant $C_1(m,r)$ such that $$\sqrt{\cO(0)^2+\cO(1)^2}\ge C_1(m,r)\norm{\cO}_{0,\cI},\qquad \calpha\in(0,\delta)\cup(1-\delta,1).$$
For $\delta \in [\delta,1-\delta]$, the following function is continuous
\begin{equation}
\mathcal{J}^{m}_{w,r}:\calpha\mapsto \frac{\cO(0)^2+\cO(1)^2}{\norm{\cO}_{0,\cI}^2} \label{eqn:definition_of_J}
\end{equation}
because both of its numerator and denominator are rational functions of $\calpha$ and the denominator is not zero.
%Since $\mathcal{J}^{m}_{w,r}$ is continuous (From \autoref{coro:J_is_rational}) and positive on $[\delta,1-\delta]$, then
Therefore, there is $C_2(m,w,r)>0$ such that
$$\sqrt{\cO(0)^2+\cO(1)^2}\ge C_2(m,w,r)\norm{\cO}_{0,\cI},\qquad \calpha\in[\delta,1-\delta].$$
By letting $C(m,w,r)=\min(C_1(m,r),C_2(m,w,r))$, we know that $\cO$ satisfies inequality \eqref{eq:norm_on_Q} stated in Lemma \ref{lem:norm_on_Q}.
Consequently, the estimates in \eqref{eq:norm_on_Q} of Lemma \ref{lem:norm_on_Q} is true for every function in $\cQ$ because it is a 
one-dimensional space, and Lemma \ref{lem:norm_on_Q} is proven. 
    
   \section{Proof of \autoref{lem:bounds_on_Hermite}}
   \label{sec:Hermite_proof}
   Let us start with $p=L^0_{\calpha,r}$, we have
    $$p(0)=1,\quad p(1)=0,\quad p'(0)=0,\quad p'(1)=0.$$
    By \autoref{thm:Zeros of IFE basis functions}, $p'$ does not change sign in $(0,1)$ since $p'\in \V^{2}_{\calpha,\tau(r)}(\cI)$ and $p'(0)=
p'(1)=0$. Therefore, $p$ is monotonically decreasing from $p(0)=1$ to $p(1)=0$. The same argument applies to $L^1_{\calpha,r}$.

    Next, we show that $q=L^2_{\calpha,r}$ is bounded between $0$ and $1$. We have $$q(0)=0,\quad q(1)=0,\quad q'(0)=1,\quad q'(1)=0.$$
    By Rolle's theorem, there is $c\in (0,1)$ such that $q'(c)=0$. By \autoref{thm:Zeros of IFE basis functions}, $c$ the only root of $q'$ in $(0,1)$. Now, by the generalized Rolle's theorem, there is $d\in (c,1)$ such that $q''(d^-)q''(d^+) \leq 0$. If $d \not= \calpha$, then $q''(d^-) = q''(d^+) = q''(d)$ because $q$ is a polynomial on either sides of $\calpha$. In this case we have $q''(d) = 0$. If $d = \calpha$, then $q''(d^-)q''(d^+) \leq 0$ and jump condition implies
\begin{align*}
\frac{\beta^-}{\beta^+} \big(q''(\calpha^-)\big)^2 \leq 0
\end{align*}    
from which we have $q''(\calpha^-) = 0 = q''(\calpha^+)$. Hence, $q''(d)=0$. Furthermore, by \autoref{thm:Zeros of IFE basis functions}, $d$ is the only root of $q''$ since $q''\in \V^1_{\calpha,\tau^2(r)}(\cI)$. Since $q''$ is a linear polynomial on either sides of $\calpha$, the jump condition satisfied by $q$ further implies that $q''$ does not change its sign $(0, d)$ and $(d, 1)$. 
Because $q'(0)=1$ and $q'(c)=0$ and $0 < c < d$, we know that $q'$ is decreasing on $(0, d)$ but increasing on $(d, 1)$. These further imply $q'(x) \in [0, 1]$ for $x \in [0, c]$ and $q'(x) \leq 0$ for $x \in [c, d]$; hence, $q(x) \leq q(c)$ for all $x \in [0, d]$. Furthermore, since $q'(d) \leq 0, q'(1) = 0$ and $q'$ is monotonic on $[d, 1]$, we know that $q'(x) \leq 0$ for all $x \in [d, 1]$. Hence, $0 = q(1) \leq q(x) \leq q(d) \leq q(c)$ for all $x \in [d, 1]$. Consequently, $q(c) \geq q(x)$ for $x \in [0, 1]$. 
In addition, since $q$ has no local minimum point on $(0, d)$, we have $q(x) \geq \min\{q(0), q(d)\} \geq 0$ for all $x \in [0, d]$. 
Thus, $q(x) \geq 0 ~\forall x \in [0, 1]$. On the other hand, 
\begin{eqnarray*}
q(x) \leq q(c) = \int_0^c q'(x)dx\le \int_0^c 1dx=c<1 ~~\forall x \in [0, 1],
\end{eqnarray*}
The last two estimates lead us to conclude that $0 \leq q(x) = L^2_{\calpha,r}(x) < 1$. As for $L^3_{\calpha,r}$, we note 
that
\begin{eqnarray*}
L^{3}_{\calpha,r}(x)=-L^{2}_{1-\calpha,\tilde{r}}(1-x),\qquad \text{where } \tilde{r}= \{r_{i}^{-1}\}_{i=0}^3
\end{eqnarray*}
which leads to $L^3_{\calpha,r}(x)\in [-1,0]$.

\commentout{
        \begin{lemma}
            Let $\beta^\pm>0$ and $\calpha \in(0,1)$, then

            $$0<L^i_{\calpha,r}(x)<1,\qquad \forall\ x\in[0,1],\quad i=0,1,2,3,$$
            where $L^i_{\calpha,r}$ is defined in \eqref{eqn:Hermite_basis}.
        \end{lemma}

        \begin{proof}
            Let us start with $p=L^0_{\calpha,r}$, we have
            $$p(0)=1,\quad p(1)=0,\quad p'(0)=0,\quad p'(1)=0.$$
            By \autoref{thm:Zeros of IFE basis functions}, $p'$ does not change sign in $(0,1)$ since $p'\in \V^{2}_{\calpha,\tau(r)}(\cI)$ and $p'(0)=p'(1)=0$. Therefore, $p$ is monotonically decreasing from $p(0)=1$ to $p(1)=0$. The same argument applies to $L^1_{\calpha,r}$.

             Next, we show that $q=L^2_{\calpha,r}$ is bounded between $0$ and $1$. We have $$q(0)=0,\quad q(1)=0,\quad q'(0)=1,\quad q'(1)=0.$$
            By Rolle's theorem, there is $c\in (0,1)$ such that $q'(c)=0$. By \autoref{thm:Zeros of IFE basis functions}, $c$ the only root of $q'$ in $(0,1)$. Therefore
            \begin{itemize}
                \item $q(c)$ is the  maximum of $q$ on $[0,1]$.
                \item $q'(x)\in [0,1]$ for $x\in [0,c]$.
                \item $q(x)> 0$ on $(0,1)$ (Because $q$ is decreasing on $[c,1]$ from $q(c)$ to $0$).

            \end{itemize}
            Using these observations and $$q(c)=\int_0^c q'(x)dx\le \int_0^c 1dx=c<1,$$
            we conclude that $q(x)\in [0,1]$ for all $x\in[0,1]$. The same argument can be applied to $L^3_{\calpha,r}$.

        \end{proof}
}

\end{appendices}

\bibliographystyle{siam}
%\bibliography{references}

\begin{thebibliography}{10}

\bibitem{adjeridAsymptoticallyExactPosteriori2010}
{\sc S.~Adjerid and M.~Baccouch}, {\em Asymptotically exact a posteriori error
  estimates for a one-dimensional linear hyperbolic problem}, Applied Numerical
  Mathematics, 60 (2010), pp.~903--914.

\bibitem{adjeridHigherDegreeImmersed2018}
{\sc S.~Adjerid, M.~{Ben-Romdhane}, and T.~Lin}, {\em Higher degree immersed
  finite element spaces constructed according to the actual interface},
  Computers \& Mathematics with Applications, 75 (2018), pp.~1868--1881.

\bibitem{adjeridImmersedDiscontinuousFinite2019}
{\sc S.~Adjerid, N.~Chaabane, T.~Lin, and P.~Yue}, {\em An immersed
  discontinuous finite element method for the {{Stokes}} problem with a moving
  interface}, Journal of Computational and Applied Mathematics, 362 (2019),
  pp.~540--559.

\bibitem{adjeridHIGHDEGREEIMMERSED2017}
{\sc S.~Adjerid, R.~Guo, and T.~Lin}, {\em {{High degree immersed finite
  element spaces by a least squares method}}}, International Journal of
  Numerical Analysis And Modeling, 14 (2017), pp.~604--625.

\bibitem{adjeridPthDegreeImmersed2009}
{\sc S.~Adjerid and T.~Lin}, {\em A {{P-th}} degree immersed finite element for
  boundary value problems with discontinuous coefficients}, Applied Numerical
  Mathematics, 59 (2009), pp.~1303--1321.

\bibitem{adjeridErrorEstimatesImmersed2020}
{\sc S.~Adjerid, T.~Lin, and Q.~Zhuang}, {\em Error estimates for an immersed
  finite element method for second order hyperbolic equations in inhomogeneous
  media}, Journal of Scientific Computing, 84 (2020).

\bibitem{adjeridHigherOrderImmersed2014}
{\sc S.~Adjerid and K.~Moon}, {\em A {{Higher Order Immersed Discontinuous
  Galerkin Finite Element Method}} for the {{Acoustic Interface Problem}}}, in
  Advances in {{Applied Mathematics}}, Springer {{Proceedings}} in
  {{Mathematics}} \& {{Statistics}}, {Cham}, 2014, {Springer International
  Publishing}, pp.~57--69.

\bibitem{barbuDifferentialEquations2016}
{\sc V.~Barbu}, {\em Differential {{Equations}}}, Springer {{Undergraduate
  Mathematics Series}}, {Springer International Publishing}, 2016.

\bibitem{benzoni-gavageMultidimensionalHyperbolicPartial2006}
{\sc S.~{Benzoni-Gavage} and D.~Serre}, {\em Multi-Dimensional Hyperbolic
  Partial Differential Equations: {{First-order}} Systems and Applications},
  {Oxford University Press}, Nov. 2006.

\bibitem{brambleEstimationLinearFunctionals1970}
{\sc J.~H. Bramble and S.~R. Hilbert}, {\em Estimation of {{linear
  functionals}} on {{Sobolev spaces}} with {{application}} to {{Fourier
  transforms}} and {{spline interpolation}}}, SIAM Journal on Numerical
  Analysis, 7 (1970), pp.~112--124.

\bibitem{brennerPolynomialApproximationTheory1994}
{\sc S.~C. Brenner and L.~R. Scott}, {\em Polynomial {{Approximation Theory}}
  in {{Sobolev Spaces}}}, in The {{Mathematical Theory}} of {{Finite Element
  Methods}}, S.~C. Brenner and L.~R. Scott, eds., Texts in {{Applied
  Mathematics}}, {Springer}, {New York, NY}, 1994, pp.~91--122.

\bibitem{caoSuperconvergenceImmersedFinite2017}
{\sc W.~Cao, X.~Zhang, and Z.~Zhang}, {\em Superconvergence of immersed finite
  element methods for interface problems}, Advances in Computational
  Mathematics, 43 (2017), pp.~795--821.

\bibitem{2021ChengZhang}
{\sc Y.~Chen and X.~Zhang}, {\em A ${P}_2$-${P}_1$ partially penalized immersed
  finite element method for stokes interface problems}, International Journal
  of Numerical Analysis and Modeling, 18 (2021), pp.~120--141.

\bibitem{ciarletFiniteElementMethod2002}
{\sc P.~G. Ciarlet}, {\em The {{Finite Element Method}} for {{Elliptic
  Problems}}}, Classics in {{Applied Mathematics}}, {Society for Industrial and
  Applied Mathematics}, Jan. 2002.

\bibitem{cockburnIntroductionDiscontinuousGalerkin1998}
{\sc B.~Cockburn}, {\em An introduction to the {{Discontinuous Galerkin}}
  method for convection-dominated problems}, in Advanced {{Numerical
  Approximation}} of {{Nonlinear Hyperbolic Equations}}, Lecture {{Notes}} in
  {{Mathematics}}, {Springer}, {Berlin, Heidelberg}, 1998, pp.~150--268.

\bibitem{guoHigherDegreeImmersed2019}
{\sc R.~Guo and T.~Lin}, {\em A higher degree immersed finite element method
  based on a {{Cauchy}} extension}, Siam Journal On Numerical Analysis, 57
  (2019), pp.~1545--1573.

\bibitem{guoImmersedFiniteElement2020}
\leavevmode\vrule height 2pt depth -1.6pt width 23pt, {\em An immersed finite
  element method for elliptic interface problems in three dimensions}, Journal
  of Computational Physics, 414 (2020), p.~109478.

\bibitem{X.He_T.Lin_Y.Lin}
{\sc X.~He, T.~Lin, and Y.~Lin}, {\em Approximation capability of a bilinear
  immersed finite element space}, Numerical Methods for Partial Differential
  Equations, 24 (2008), pp.~1265--1300.

\bibitem{heImmersedFiniteElement2013}
{\sc X.~He, T.~Lin, Y.~Lin, and X.~Zhang}, {\em Immersed finite element methods
  for parabolic equations with moving interface}, Numerical Methods for Partial
  Differential Equations, 29 (2013), pp.~619--646.

\bibitem{2021JonesZhang}
{\sc D.~Jones and X.~Zhang}, {\em A class of nonconforming immersed finite
  element methods for {Stokes} interface problems}, Journal of Computational
  and Applied Mathematics, 392 (2021).

\bibitem{liImmersedInterfaceMethod1998}
{\sc Z.~Li}, {\em The immersed interface method using a finite element
  formulation}, Applied Numerical Mathematics, 27 (1998), pp.~253--267.

\bibitem{liImmersedFiniteElement2004}
{\sc Z.~Li, T.~Lin, Y.~Lin, and R.~Rogers}, {\em An immersed finite element
  space and its approximation capability}, Numerical Methods for Partial
  Differential Equations, 20 (2004), pp.~338--367.

\bibitem{liNewCartesianGrid2003b}
{\sc Z.~Li, T.~Lin, and X.~Wu}, {\em New cartesian grid methods for interface
  problems using finite element formulation}, Numerische Mathematik, 96 (2003),
  pp.~61--98.

\bibitem{linErrorAnalysisImmersed2017}
{\sc M.~Lin, T.~Lin, and H.~Zhang}, {\em Error analysis of an immersed finite
  element method for {{Euler-Bernoulli}} beam interface problems},
  International Journal of Numerical Analysis And Modeling, 14 (2017),
  pp.~822--841.

\bibitem{linRectangularImmersedFinite2001}
{\sc T.~Lin, Y.~Lin, R.~Rogers, and M.~L. Ryan}, {\em A rectangular immersed
  finite element space for interface problems}, Adv. Comput. Theory Pract, 7
  (2001), pp.~107--114.

\bibitem{linImmersedFiniteElement2011}
{\sc T.~Lin, Y.~Lin, W.-W. Sun, and Z.~Wang}, {\em Immersed finite element
  methods for 4th order differential equations}, Journal of Computational and
  Applied Mathematics, 235 (2011), pp.~3953--3964.

\bibitem{linImmersedFiniteElement2013}
{\sc T.~Lin, Y.~Lin, and X.~Zhang}, {\em Immersed finite element method of
  lines for moving interface problems with nonhomogeneous flux jump}, Vertex
  Operator Algebras and Related Areas, 586 (2013), pp.~257--265.

\bibitem{2019LinLinZhuan_Helmholtz1}
{\sc T.~Lin, Y.~Lin, and Q.~Zhuang}, {\em Solving interface problems of the
  {Helmholtz} equation by immersed finite element methods}, Communications on
  Applied Mathematics and Computation, 1 (2019), p.~187–206.

\bibitem{linOptimalErrorBounds2020}
{\sc T.~Lin and Q.~Zhuang}, {\em Optimal error bounds for partially penalized
  immersed finite element methods for parabolic interface problems}, Journal of
  Computational and Applied Mathematics, 366 (2020).

\bibitem{lombardModelisationNumeriquePropagation2002}
{\sc B.~Lombard}, {\em {Mod\'elisation num\'erique de la propagation des ondes
  acoustiques et \'elastiques en pr\'esence d'interfaces}}, PhD thesis,
  Universit\'e de la M\'editerran\'ee - Aix-Marseille II, Jan. 2002.

\bibitem{Lombard_Piraux_2001_1D}
{\sc B.~Lombard and J.~Piraux}, {\em A new interface method for hyperbolic
  problems with discontinuous coefficients: one-dimensional acoustic example},
  Journal of Computational Physics, 1168 (2001), pp.~227--248.

\bibitem{moonImmersedDiscontinuousGalerkin2016}
{\sc K.~Moon}, {\em Immersed {{Discontinuous Galerkin Methods}} for {{Acoustic
  Wave Propagation}} in {{Inhomogeneous Media}}}, Ph.D. thesis, Virginia Tech,
  (2016).

\bibitem{SVallaghe_TPapadopoulo_TriLinear_IFE}
{\sc S.~Vallagh\'{e} and T.~Papadopoulo}, {\em A trilinear immersed finite
  element method for solving the electroencephalography forward problem}, SIAM
  J. Sci. Comput., 32 (2010), pp.~2379--2394.

\bibitem{2022WangZhangZhuang}
{\sc J.~Wang, X.~Zhang, and Q.~Zhuang}, {\em An immersed crouzeix-raviart
  finite element method for navier-stokes interface problems. an immersed
  crouzeix-raviart finite element method for navier-stokes interface
  problems.}, International Journal of Numerical Analysis \& Modeling, 19
  (2022), pp.~563--586.

\bibitem{wangHermiteCubicImmersed2005}
{\sc T.~S. Wang}, {\em A {{Hermite}} Cubic Immersed Finite Element Space for
  Beam Design Problems}, PhD thesis, Virginia Tech, 2005.

\bibitem{yangAnalysisOptimalSuperconvergence2012}
{\sc Y.~Yang and C.-W. Shu}, {\em Analysis of optimal superconvergence of
  discontinuous galerkin method for linear hyperbolic equations.}, SIAM Journal
  on Numerical Analysis, 50 (2012), pp.~3110--3133.

\end{thebibliography}

\end{document}